\documentclass[12pt]{amsart}
 \usepackage{amsmath, amssymb, xypic, verbatim, color, amsfonts, amsthm,amscd}
\usepackage{graphicx}
\usepackage{colordvi}
\usepackage{times}

\swapnumbers \theoremstyle{plain}
\newtheorem{theorem}{Theorem}[section]
\newtheorem{remark}[theorem]{Remark}

\newtheorem{corollary}[theorem]{Corollary}

\newtheorem{proposition}[theorem]{Proposition}

\newtheorem{lemma}[theorem]{Lemma}

\newtheorem{definition}[theorem]{Definition}

\theoremstyle{definition}

\newtheorem{example}[theorem]{\bf Example}

  \def\simto{\overset{\sim}{\longrightarrow}}
  
  \def\ip<#1>{\langle#1\rangle}   

\newcommand{\al}{\alpha}  
\newcommand{\gam}{\gamma}  \newcommand{\lam}{\lambda}

 \renewcommand{\tilde}{\widetilde}
  \def\b1{\text{\bf\large 1}}

\newcommand{\Aut}{\operatorname{Aut}} 
 
\newcommand{\End}{\operatorname{End}} 
\newcommand{\Hom}{\operatorname{Hom}} \newcommand{\Imo}{\operatorname{Im}}

\newcommand{\perm}{\operatorname{perm}} 
 
 \newcommand{\GL}{\operatorname{GL}}
\newcommand{\SL}{\operatorname{SL}}
\newcommand{\Spec}{\operatorname{Spec}}

\newcommand{\beqn}{\begin{equation}} \newcommand{\eeqn}{\end{equation}}

  \newcommand{\fp}{\mathsf{p}}
  
 \newcommand{\fv}{\mathfrak{v}}

\newcommand{\bc}{\mathbb{C}}  
 \newcommand{\bz}{\mathbb{Z}}
\def\simto{\overset{\sim}{\longrightarrow}}

\newcommand{\cz}{\mathcal{Z}}
\newcommand{\cy}{\mathcal{Y}}
\newcommand{\cx}{\mathcal{X}}

\begin{document}

\title[Geometry of Permanents]{Geometry of orbits of permanents and determinants}
 \author[Shrawan Kumar]{Shrawan Kumar}
 \maketitle
 \section{Introduction}

 Let  ${\mathfrak v}$ be a complex vector space of dimension $m$  and let $E:={\mathfrak v}\otimes
{\mathfrak v}^* = \End {\mathfrak v}$.  Consider $\det\in Q:=S^m(E^*)$, where $\det$ is the function taking
determinant of any $X\in \End {\mathfrak v}$.
Fix a basis
 $\{ e_1, \dots, e_m \}$ of ${\mathfrak v}$ and a positive integer $n<m$ and
 consider the
 function $\fp\in Q$, defined by $\fp(X) =
x^{m-n}_{1,1} \perm (X^o),$ $X^o$ being the component of $X$ in the right
 down $n\times n$
corner, where any element of $\End \fv$ is represented by a $m\times m$-matrix
$X=(x_{i,j})_{1\leq i,j,\leq m}$ in the basis $\{e_i\}$ and perm denotes the
permanent. The group $G=\GL(E)$
canonically acts on $Q$. Let $\cx$ (resp. $\cy$) be the $G$-orbit closure of
det (resp. $\fp$) inside $Q$. Then, $\cx$ and $\cy$ are closed (affine)
subvarieties
of $Q$ which are stable under the standard homothecy action of $\bc^*$ on $Q$.
Thus, their affine coordinate rings $\bc[\cx]$ and $\bc[\cy]$ are
nonnegatively graded $G$-algebras over the complex numbers $\bc$. Clearly,
$\End E\cdot \det\subset \cx$, where $\End E$ acts on $Q$ via:
$(g\cdot q)(X)=q(g^t\cdot X)$, for $g\in \End E, q\in Q$ and
$X\in E$.

For any positive integer $n$, let $\bar{m}=\bar{m}(n)$ be the smallest
positive integer such that the permanent of any $n\times n$ matrix can be
realized as a linear projection of the determinant of a $\bar{m}\times \bar{m}$
matrix. This is equivalent to saying that $\fp \in \End E\cdot \det$ for the pair
$(\bar{m},n)$. Then,  Valiant conjectured that the function $\bar{m}(n)$ grows faster than
any polynomial in $n$ (cf. [V]).

Similarly, let $m=m(n)$ be the smallest integer such that $\fp\in \cx$
(for the pair $(m,n)$). Clearly, $m(n)\leq \bar{m}(n)$. Now,
Mulmuley-Sohoni strengthened Valiant's conjecture. They conjectured that,
in fact, the function $m(n)$ grows faster than
any polynomial in $n$ (cf. [MS1], [MS2] and the references therein). They further conjectured that
if $\fp\notin \cx$, then there exists an irreducible $G$-module which occurs in
$\bc[\cy]$ but does not occur in $\bc[\cx]$. (Of course, if $\fp\in \cx$,
then $\bc[\cy]$ is a $G$-module quotient of $\bc[\cx]$.) This Geometric Complexity Theory
programme initiated by Mulmuley-Sohoni provides a significant mathematical approach to
solving the Valiant's conjecture (in fact, strengthened version of
Valiant's conjecture proposed by them). In a recent paper, Landsberg-Manivel-Ressayre
[LMR] have shown that $m(n)\geq n^2/2$.

It may be remarked that, since $(\perm_n)_{n\geq 1}$ is $\bf{VNP}$-complete
(cf. [V]),  Valiant's above conjecture is equivalent to
$(\perm_n)_{n\geq 1}\notin \bf{VP}$. This is an algebraic version of
Cook's celebrated $\bf{P}\neq \bf{NP}$ conjecture. The conjecture of
Mulmuley-Sohoni is equivalent to $(\perm_n)_{n\geq 1}\notin
\overline{\bf{VP}_{ws}}$. For a survey of these problems, we refer to
the article [BL] by Burgisser-Landsberg-Manivel-Weyman.

 From the experience in representation theory (e.g., the Demazure character
 formula or the study of functions on the nilpotent cone), one important property of
 varieties which allows one to study the ring of regular functions on them
 is their {\em normality}. But, unfortunately, as we show in the paper,
 both of the varieties $\cx$ (for any $m\geq 3$) and $\cy$ (for any $m\geq 2n$ and
 $n\geq 3$) are {\em not} normal (cf. Theorems
 ~\ref{2.6} and ~\ref{7.4}).

 To prove the nonnormality of $\cx$, we study the defining equations of
 the boundary $\partial \cx:=\cx\setminus \cx^o$ and show that there exists
 a $G'$-invariant $f_o$ in $\bc[\cx]$ (where $G':=\SL(E)$), which defines  $\partial \cx$ set theoretically
 (but not scheme theoretically), cf. Corollaries ~\ref{2.5} and ~\ref{2.7}.
 In particular, each irreducible component of $\partial \cx$ is of
 codimension one in
 $\cx$ (cf. Corollary ~\ref{2.5}).
 To show that $\cx$ is not normal, we show that, in fact, the GIT quotient
 $\cx':=\cx//G'$ is not normal by analyzing the $G'$-invariants in
 $\bc[\cx]$.

 Let $\{ e_1^*, \dots,
 e_m^* \}$ be the dual basis of
${\mathfrak v}^*$. Then, of course, $\{e_{i,j}:=e_i\otimes e_j^*; 1
\leq i,j\leq m\}$ is a basis of $E$.
Let $S_1$ be the subspace of $E$  spanned by $\{e_{i,j}; m-n+1 \leq i,j\leq m
\}$, $S$ the subspace  of $E$  spanned by $S_1$ and
$e_{1,1}$ and $S^{\bot}$  the complementary
subspace spanned by the set $\{e_{i,j}\}_{1\leq i,j,\leq m}\setminus
\{e_{1,1}, e_{i,j}\}_{m-n+1 \leq i,j\leq m}$. Let $P$ be the maximal parabolic subgroup of
$G=\GL(E)$ which keeps the subspace $S^{\bot}$ of $E$ stable and let  $L_P$ be the Levi subgroup of $P$ defined by:
 $L_P=L_P^1\times L_P^2$, where $L_P^1:=\GL(S^{\bot})$ and $L_P^2:=\GL(S)$.
Let $R$ be the parabolic subgroup of $L_P^2$ which fixes the line
spanned by $e_{1,1}$.

 The proof of the nonnormality of $\cy$ is more involved. We first show that
 the $G$-module decomposition of $\bc[\cy]$ is equivalent to the $L_P^2$-module
 decomposition
 of the ring of the regular functions on the $L_P^2$-orbit closure
 $\mathcal{C}$ of
 $\fp$ (cf. Theorem ~\ref{4.2}). Next, we analyze $\mathcal{C}$
 in Section 6. In particular, we give its partial
 desingularization of the form $\mathcal D:=L_P^2\times_R \,
 (S^*\times \cz//\bc^*)$ (cf. Proposition ~\ref{5.3} and
 Lemma ~\ref{5.2}), where $\cz$ is the $\GL(S_1)$-orbit closure
 of the permanent function perm inside $S^n(E^*)$, $\bc^*$ acts on
 $S^*\times \cz$ via the equation \eqref{e5.1} and the action
 of $R$ on $S^*\times \cz//\bc^*$ is given in Section 6 immediately after
 Lemma ~\ref{5.2}. We determine the ring
 of regular functions on $\mathcal D$ (as a $L_P^2$-module) completely
 (and explicitly) in terms of the ring of regular functions on $\cz$ as a
 $\GL(S_1)$-module
 (cf. Theorem ~\ref{6.5}). Via the Zariski's main theorem, this allows one to give the $G$-module
 decomposition of the normalization of $\cy$ completely in terms of the
 $\GL(S_1)$-module decomposition of the ring of regular functions on the
 normalization of the $\GL(S_1)$-variety $\cz$ (use Theorem ~\ref{4.2}, Corollary ~\ref{4.4},
 Lemma ~\ref{5.2}, Proposition \ref{5.3} and Theorem ~\ref{6.5}). It may be remarked that we are not
 able to give an explicit $G$-module decomposition of $\bc[\cy]$ itself
 from that of the $\GL(S_1)$-module $\bc[\cz]$. By comparing the explicit $L_P^2$-module
 decomposition of
 the ring of regular functions $\bc[\mathcal D]$ mentioned above
 with
 the ring of regular functions on the $L_P^2$-orbit closure of
 $\fp$,  we conclude that
 $\cy$ is not normal for any $m\geq 2n$ and $n\geq 3$ (cf. Theorem ~\ref{7.4}). A similar idea
 allows us to conclude that the orbit closures of $\fp$ under the groups
 $R$ and $L_P^2$ are not normal (cf. Corollaries ~\ref{7.2} and ~\ref{7.5}).

 \noindent
 {\em Acknowledgements.}  I thank J. Landsberg for bringing my attention to
 the works of Mulmuley-Sohoni and his comments to an earlier version of the paper and to K. Mulmuley for explaining to me some of his
 works. This work was partially supported by the NSF grant DMS 0901239.

\section{Coordinate ring of the orbit closure of det}

Take a vector space ${\mathfrak v}$ of dimension $m$ and let $E={\mathfrak v}\otimes
{\mathfrak v}^* = \End {\mathfrak v}$. Consider $G = \GL(E)$ acting canonically on
$Q=S^m(E^*)$, and consider $\det\in Q$, where $\det$ is the function taking
determinant of any $A\in \End {\mathfrak v}$.

Recall the following result due to Frobenius [Fr] (cf., e.g., [GM] for a survey).
\begin{proposition} \label{1.1} The isotropy $G_{\det} \subset G$
consists of the transformations of the form $\tau : Y\mapsto AY^*B,$ where $Y^* = Y$ or $Y^t$
and $A,B\in \SL({\mathfrak v})$. (Here $Y^t$ denotes the transpose of $Y$
with respect to a fixed basis of  $\mathfrak v$.)
\end{proposition}

\begin{lemma}\label{1.2} Any $\tau$ of the form $\tau(Y) = AYB$ as above can be written as
\beqn
\End {\mathfrak v} = {\mathfrak v}\otimes {\mathfrak v}^* \to {\mathfrak v}\otimes
 {\mathfrak v}^*,\,\, v\otimes f \;\;\; \mapsto Av\otimes
B^*f,\eeqn
where $B^*$ is the dual map induced from $B$.
 In particular, such a $\tau$ has determinant 1.

 If $\tau$ is of the form $\tau(Y) = AY^tB$ as in the above proposition, then
 \beqn\det \tau = (-1)^{\frac{m(m-1)}{2}}.\eeqn
 \end{lemma}

\begin{proof} Take a basis $\{ e_i\}$ of ${\mathfrak v}$ and let $\{ e^*_i\}$ be the
 dual basis of
${\mathfrak v}^*$. Let $A=(a_{i,j})$ be the matrix of $A$ in the basis $\{e_i\}$
of $\mathfrak v$ and similarly $B=(b_{i,j})$. Then, \[ (B^*e_j^*)\, e_p = e^*_j\, (Be_p)
 = \sum_{\ell}\, e^*_j \bigl( b_{\ell,
p}\, e_{\ell}\bigr)= b_{j,p} .\] Thus, $ B^*e^*_j = \sum_p\, b_{j,p}\, e^*_p .$ Hence,
denoting the map (1) by $\hat{\tau}$, we have
 \[
e_{i,j}:= e_i\otimes e^*_j \overset{\hat{\tau}}{\longmapsto} Ae_i\otimes B^*(e^*_j) =
\sum_{k,p} a_{k,i}e_k\otimes b_{j,p}e^*_p = \sum_{k,p} a_{k,i}b_{j,p}\, e_k\otimes e^*_p .\]
Thus, \[\bigl( \hat{\tau} (e_{i,j})\bigr)_{k,p}= a_{k,i} b_{j,p}  = (A e_{i,j}B)_{k,p} ,\]
where $\bigl( \hat{\tau} (e_{i,j})\bigr)_{k,p}$ denotes the $(k,p)$-th component of
$\hat{\tau} (e_{i,j})$ in the basis $\{e_{k,p}\}$. This
proves $\tau = \hat{\tau}$.

Let $\{\lam_1,\dots ,\lam_m\}$ be the eigenvalues of $A$ and $\{ \mu_1,\dots ,\mu_m\}$ the
eigenvalues of $B$. Then, \begin{align*} \det \hat{\tau} &= \prod^{m}_{i,j=1} \lam_i\mu_j\\
&= \prod_i\, \bigl(\lam^m_i \det B\bigr)\\ &= (\det A)^m\, (\det B)^m \\ &= 1, \; \text{
since $\det A = \det B=1$}. \end{align*}

To prove (2), in view of the above, we can assume that $\tau(Y)=Y^t$. The proof in this case
 is easy. \end{proof}

As a consequence  of Proposition ~\ref{1.1} and Lemma ~\ref{1.2}, we get the following.
\begin{corollary}\label{1.3}
We have a group isomorphism:
 \[ \phi: \SL({\mathfrak v}) \times \SL({\mathfrak v})/\Theta_m\simeq G^o_{\det}, \,\, \phi
 [A,B]
(v\otimes f)= Av\otimes (B^{-1})^*f, \] where $\Theta_m$ is the group of the $m$-th
roots of unity acting on $ \SL({\mathfrak v}) \times \SL({\mathfrak v})$ via: $z(A,B)=(zA,zB)$,
$[A,B]$ denotes the $\Theta_m$-orbit of $(A,B)$ and
$ G^o_{\det}$ denotes the identity component of $ G_{\det}$.

In particular, $\dim (G'\cdot \det)=(m^2-1)^2$, where $G':=\SL(E)$. Moreover,
$G^o_{\det} \subset G'_{\det}$.

If $\left(\begin{smallmatrix} m \\ 2
\end{smallmatrix}\right)$ is even, then $G_{\det} \subset G'$.
\end{corollary}

Since the isotropy $G'_{\det}$ is not contained in any
proper parabolic subgroup of $G'$ (as can be easily seen by observing that no
proper subspace of $E$ is stable under $G^o_{\det}$), Kempf's theorem [Ke, Corollary 5.1]
gives the following result observed in [MS1, Theorem 4.1]:
\begin{proposition} \label{1.4} The orbit $G'\cdot \det$ is closed in $Q$.
\end{proposition}

 Let $\cx^o:=G\cdot \det, \cx:=\overline{\cx^o},$ where the closure is taken inside $Q$, and let
 $\cx':=G'\cdot \det.$ The following simple lemma is taken from [MS].
\begin{lemma} \label{1.5} For any $d\geq 0$, the restriction map
\[ \phi^d : \bc^d[\cx] \to \bc [\cx'] \] is
injective, where $\bc^d[\cx]$ is the homogeneous degree $d$-part of
 $\bc [\cx]$ (i.e., $\bc^d[\cx]$ is a quotient of $S^d(S^m(E))$).
\end{lemma}

\begin{proof} Take $f\in\bc^d[\cx]$ such that $\phi^d(f) =0$, i.e.,
$f(x)=0$, for all $x\in \cx'$. Then, for any $ z\in\bc$ and $x\in \cx', f(zx) = z^d f(x) =
 0$, i.e., $ f(\bc\cdot \cx')\equiv
0$ and hence $f(\overline{\bc\cdot \cx'}) \equiv 0$. But,
$\overline{\bc\cdot \cx'}=\cx$ and hence  $f(\cx) \equiv 0.$ This proves the
lemma. \end{proof}

As a consequence of Proposition ~\ref{1.4} and Lemma ~\ref{1.5} and the
Frobenius reciprocity, one has the
following result due to [MS2]:
\begin{corollary} \label{1.6} An irreducible $G'$-module $M$ occurs in
$\bc [G'/G'_{\det} ] = \bc [\cx']
\Leftrightarrow M$ occurs in $\bc[\cx]$. In particular, an
irreducible $G'$-module $M$ occurs in $\bc[\cx]
\Leftrightarrow M^{G'_{\det}} \neq 0$. \end{corollary}

\begin{example} Let $m=2$. Then, $G\cdot \det$ is dense in $Q=S^2(E^*)$
(since they have the same dimensions by Corollary ~\ref{1.3}). Moreover,
$Q$ has $5$ orbits under $G$ of dimensions: $10, 9, 7, 4, 0$.

To show this, obeserve that there are exactly $5$ quadratic forms in $4$ variables (up to
the change of a basis): $x_1^2+x_2^2+x_3^2+x_4^2; x_1^2+x_2^2+x_3^2; x_1^2+x_2^2;
x_1^2;0$. Their isotropies under the $G$-action have dimensions:
$6,7,9,12,16$ respectively.
\end{example}

\section{Non-normality of the orbit closure of $\det$}

We first recall the following two elementary lemmas from commutative Algebra.
\begin{lemma} \label{2.1} Let $R$ be a $\Bbb Z_+$-graded algebra over the complex numbers $\Bbb C$ with
the degree $0$-component $R^0=\Bbb C$ and let $M$ be a $\Bbb Z_+$-graded $R$-module. Let
$\mathfrak m$ be the augmentation ideal $\oplus_{d>0}\, R^d$ and assume that
$M/(\mathfrak m\cdot M)$ is a finite dimensional vector space over $R/\mathfrak m\simeq\Bbb C.
$ Then, $M$ is a finitely generated $R$-module.
\end{lemma}
\begin{proof} Choose a set of homogeneous generators $\{\bar{x}_1, \dots, \bar{x}_n\}
\subset M/(\mathfrak m\cdot M)$ over $R/\mathfrak m$ and let $x_i\in M$ be
a homogeneous lift of $\bar{x}_i$. Let $N\subset M$ be the graded
$R$-submodule: $Rx_1+\dots +Rx_n$. It is easy to see that \beqn
\label{e2.1}\mathfrak m\cdot (M/N)=M/N.\eeqn If $M/N \neq 0$, let
$d_o\geq 0$ be the smallest degree such that $(M/N)^{d_o}\neq 0$.
Clearly, \eqref{e2.1} contradicts this. Hence $N=M$.
\end{proof}
\begin{lemma} \label{2.2} Let $R$ and $S$ be two non-negatively graded finitely generated domains over
$\Bbb C$ such that $R^0=S^0=\Bbb C$  and let $f:R \to S$ be a graded
algebra injective homomorphism. Assume that the induced map
$\hat{f}: \Spec S \to \Spec R$ satisfies $(\hat{f})^{-1} (\mathfrak
m_R)=\{\mathfrak m_S\}$, where $\mathfrak m_S$ is the augmentation ideal
of $S$ and $\Spec S$ denotes the space of maximal ideals of $S$.
Then, $S$ is a finitely generated $R$-module; in particular, it is
integral over $R$.
\end{lemma}
\begin{proof} Let $\mathfrak m_R'$ be the ideal in $S$ generated by $f(\mathfrak m_R)$. Then,
by assumption, $\mathfrak m_S$ is the only maximal ideal of $S$ containing
$\mathfrak m_R'$. Hence, the
radical ideal $\surd{\mathfrak m_R'}=\mathfrak m_S.$ Thus,
$\mathfrak m_R'\supset \mathfrak m_S^d$, for some $d>0$ (cf. [AM,
Corollary 7.16]). In particular, $S/\mathfrak m_R'$ is a finite
dimensional vector space over $\Bbb C$ and hence by the above lemma,
$S$ is a finitely generated $R$-module. This proves that $S$ is
integral over $R$ (cf. [AM, Proposition 5.1]).
\end{proof}

Let
$\partial \cx:=\cx\setminus \cx^o$ be its boundary; all equipped with the
locally-closed (reduced) subvariety structure coming from $Q$. Let $\mathcal{I}\subset
\Bbb C[\cx]$
denote the ideal of $\partial \cx$.

\begin{lemma} \label{2.3}
For any nonzero $G$-submodule $V\subset \mathcal{I}$, the zero set
\[ Z(V):=\{y\in \cx: f(y)=0 \,\forall f\in V\}\]
equals $\partial \cx$.
\end{lemma}
\begin{proof} Of course, $Z(V)\supset \partial \cx$. Moreover, $Z(V)$ is a
$G$-stable subset of $\cx$. If $Z(V)$ properly contains $\partial \cx$,
then $Z(V)=\cx,$ which is a contradiction since $V$ is  nonzero.
\end{proof}
\begin{remark} {\em The above lemma is clearly true (by the same proof)
for any $G$-orbit closure $X$ in an affine $G$-variety $Y$.}
\end{remark}

\begin{proposition} \label{2.4} The ideal $\mathcal{I}\subset \Bbb C[\cx]$  contains
 a nonzero $G'$-invariant.
\end{proposition}
\begin{proof} Let $Z:=\cx//G'$, where (as earlier) $G'= \SL(E)$. Then,
$Z$ is
an irreducible affine variety with $\Bbb C^*$-action
coming from the action of $\Bbb C^*$ on $Q$ via: $z\cdot v=z^mv$.
Consider the $\Bbb C^*$-equivariant map $\sigma: \Bbb C \to \cx$,
$z\mapsto (zI)\odot \det,$ where ($(zI)\odot \det ) (e)= \det (ze),$
for any $e\in E$,  and $\Bbb C^*$ acts on $\Bbb C$ via: $z\cdot v=zv$.
Consider the composite map $\bar{\sigma}=\pi\circ \sigma: \Bbb C \to Z$,
 where $\pi:\cx \to \cx//G'$ is
the canonical projection. By Proposition ~\ref{1.4}, $(\bar{\sigma})^{-1}
\{0\}=\{0\}.$ Moreover, clearly $\bar{\sigma}$ is a dominant morphism since
$G\cdot \det$ is dense in $\cx$. Thus, by Lemma ~\ref{2.2}, $\bar{\sigma}$
is a finite (in particular, surjective) morphism. Moreover,
  no $G'$-orbit $S$ in $\partial
\cx\setminus \{0\}$ is closed in $\cx$. In fact, for any such $S$, $0\in
\bar{S}$:

Let $S'$ be a closed $G'$-orbit in $\bar{S}$. If $S'$ is nonzero,
 $S'=G'\cdot\sigma (z)$, for some $z\in \Bbb C^*$, since $\bar{\sigma}$ is surjective.
 But, $G'\cdot \sigma(z)\subset \cx^o$, whereas $S'\subset \partial \cx.$ This is a
 contradiction. Hence $0\in \bar{S}$.

Take any nonzero homogeneous polynomial $f_o\in \Bbb C[Z]=\Bbb C[\cx]^{G'}$ of
positive degree. Then, $f_o$ restricted to $\partial \cx//G'$ is identically zero,
since $\partial \cx//G'\simeq \{0\}.$ Hence, $f_o\in \mathcal{I}$. This proves the lemma.

\end{proof}
\begin{corollary} \label{2.5} For any nonzero homogeneous $f_o\in \Bbb C[\cx]^{G'}$ of positive
 degree, the zero set $Z(f_o)=\partial \cx$. In particular,
\[\surd\langle f_o\rangle = \mathcal I,\]
where $\langle f_o\rangle $ is the ideal of $\Bbb C[\cx]$ generated by
$f_o$.

Moreover, each irreducible component of $\partial \cx$ is of codimension one in $\cx$.
\end{corollary}
\begin{proof} By the last paragraph of the proof of the above proposition,
${(f_o)}_{|\partial \cx}\equiv 0$.  Thus, the first part of the corollary is
a particular case of Lemma ~\ref{2.3}.

 For
 the second part, observe that $f_o$ does not vanish anywhere on $\cx^o$
 since $f_o$ is $G'$-invariant and homogeneous. Moreover,
 $f_o\circ \bar{\sigma}: \Bbb C \to \Bbb C$ is surjective
 (being nonzero).  Now use
 [S, Theorem 7, page 76].

 \end{proof}
\begin{remark} {\em (a) The assertion in the above corollary, that each irreducible
component of $\partial \cx$ is of codimension one in $\cx$, can also be proved
by using Lemma ~\ref{4.6a}. (Observe that $G\cdot\det$ is affine by
Corollary ~\ref{1.3}, using
Matsushima's theorem.)

(b) Let $V$ be a nontrivial irreducible representation of $\GL(d)$ and let
$v_o\in V$ be such that $\SL(d)$-orbit of $v_o$ is closed. Then, it is easy to
see (by the same proof) that Lemma ~\ref{1.5}, Proposition ~\ref{2.4} and
Corollary ~\ref{2.5} remain true for the $\GL(d)$-orbit closure $X$ of
$v_o$.}
\end{remark}

\begin{theorem} \label{2.6} For any $m\geq 3$, $\cx=\overline{G\cdot \det}$ is not normal.
\end{theorem}
\begin{proof} Assume that $\cx$ is normal, then so would be
$Z=\cx//G'$. By Matsushima's theorem, since the isotropy of
$\det$ is reductive (cf. Corollary ~\ref{1.3}), $\cx^o$ is an affine
variety. By the Frobenius reciprocity, \beqn \bc
[\cx^o]^{G'}\simeq \oplus_{n\in \bz}\, V(n\delta)\otimes
[V(n\delta)^*]^{G_{\det}}, \eeqn where $V(n\delta)$ is the
irreducible $G$-module with highest weight corresponding to the
partition $(n\geq \dots \geq n)$ ($m^2$ factors). By Lemma
~\ref{1.2}, if $m(m-1)/2$ is even, $[V(n\delta)^*]^{G_{\det}}$ is
one dimensional, for all $n\in \bz$. If  $m(m-1)/2$ is odd,
\begin{align}
\dim [V(n\delta)^*]^{G_{\det}}&=1, \,\,\text{if $n$ is even}\\
&=0, \,\,\text{if $n$ is odd}.
\end{align}
For $d\in \Bbb Z_+$, let $\bc^d[\cx^o]$ denote the subspace of $\bc[\cx^o]$
such that, for any $z\in \bc^*$, the matrix $zI$ acts via $z^{md}$. Let $\hat{f}_o\in \bc^{p_mm}[\cx^o]^{G'}$ be a nonzero
element,  where $p_m=1$ if $m(m-1)/2$ is even and $p_m=2$ if
$m(m-1)/2$ is odd. Then, clearly,
\[\bc^{\geq 0}[\cx^o]^{G'}\simeq \oplus_{n\in \bz_+}\,\bc
\hat{f}_o^n.\] Now, $\bc[\cx]^{G'}\subset \bc[\cx^o]^{G'}$ is
a homogeneous subalgebra. Let $d_o>0$ be the smallest integer such
that $f_o=\hat{f}_o^{d_o}\in \bc[\cx]^{G'}$. (Such a $d_o$ exists by Proposition
~\ref{2.4}.) Since, by
assumption, $\bc[\cx]^{G'}$ is a normal ring, $\hat{f}_o\in
\bc^{p_mm}[\cx]^{G'}$. In particular, from the surjectivity
$\bc[Q]\twoheadrightarrow \bc[\cx]$, we would get
$\bc^{p_mm}[Q]^{G'}\neq 0,$ hence $S^{p_mm}(Q^*)^{G'}\neq
0.$ This contradicts [Ho, Proposition 4.3(a)]. Thus, $Z$ (and hence
$\cx$) is not normal.
\end{proof}

\begin{corollary} \label{2.7} For any $m\geq 3$,
and any nonzero homogeneous $f_o\in \Bbb C[\cx]^{G'}$ of positive
 degree,
$\langle f_o\rangle$ is not a radical ideal of $\bc[\cx].$
\end{corollary}
\begin{proof} Let $\bc(\cx)=\bc(\cx^o)$ be the function field of $\cx$ (or $\cx^o$). As in the proof of the above theorem,
$\cx^o$ is affine and, of course, normal (in fact, smooth). Take a
function $h\in \bc (\cx)$ which is integral over $\bc[\cx]$. Since $\cx^o$
is normal, $h\in \bc[\cx^o]$. If $h\notin \bc[\cx]$, we can write
$h=h_1/f_o^{d_o}$, for some $d_o> 0$ and $h_1\in \bc[\cx]\setminus
\langle f_o\rangle$ (cf. [S, Page 50] and Corollary ~\ref{2.5}).
From this (and since $h$ is integral over $\bc[\cx]$) we see that
$h_1^d\in \langle f_o\rangle$, for some $d> 0$. If $\langle
f_o\rangle$ were a radical ideal, we would have $h_1 \in
\langle f_o\rangle$. This contradicts the choice of $h_1$. Hence
$h\in \bc[\cx]$. Thus, $\cx$ is normal, contradicting Theorem
~\ref{2.6}. This proves the corollary.
\end{proof}
\begin{remark} {\em The saturation property fails  for
$\bc[\cx]$ for $m=2$.

By [GW, Page 296], as modules for $\GL(d)$ (for any $d\geq 1$),
\[ S(S^2(\bc^d))\simeq \oplus_{\mu \in 2\sum_{i=1}^d\bz_+\omega_i}\,V(\mu),\]
where $\omega_i:=\epsilon_1+ \cdots +\epsilon_i$ is the $i$-th fundamental weight of $\GL(d)$.
Observe that, for $m=2$, since $\cx=Q$, we have $\bc[\cx]=S(S^2(E))$.
Thus, $V(2\omega_2)$ appears in $S^2(S^2(E))$, but $V(\omega_2)$
does not appear in $S^1(S^2(E))$.}
\end{remark}

\section{Isotropy of Permanent}

Consider the space $\fv$ of dimension $m$ as in Section 1. Fix a positive integer
$n<m$. Choose a basis
 $\{e_{1}, \dots ,e_m\}$ of $\fv$ and consider the subspace $\fv_1$ of dimension $n$
 spanned by $\{e_{m-n+1}, \dots, e_m\}$.
 We identify $\End \fv_1$ with the space of $n\times n$-matrices (under the basis
 $\{e_{m-n+1}, \dots, e_m\}$). Then, the {\it permanent}  of a $n\times n$-matrix
 gives rise to the function $\perm \in S^n((\End \fv_1)^*).$ Consider the standard action of
 $\GL(\End \fv_1)$ on $S^n((\End \fv_1)^*)$. In particular, $\GL(\End \fv_1)$ acts on
 $\perm$.

Recall the following from [MM] (cf. also [B]).
\begin{proposition} \label{3.1} For $n \geq 3$, the isotropy of perm under the action of the
group $\GL(\End \fv_1)$ consists of the transformations
 \[ \tau : X\mapsto \lam X^* \mu , \] where $X^*$ is $X$ or $X^t$ and $\lam
,\mu$ belong to the subgroup $\hat{D}$ of
$\GL(\fv_1)$ generated by the permutation matrices together with  the diagonal matrices of
determinant 1.
\end{proposition}
Lemma ~\ref{1.2} and its proof give the following.
\begin{lemma} \label{3.1a} The determinant of the above map $\tau : X\mapsto \lam X^*\mu$ is given by
\begin{align*} \det \tau &= (-1)^{\frac{n(n-1 )}{2}}\, (\det \lam)^n\, (\det \mu)^n,
\quad\text{if }X^*=X^t,\\ &= (\det \lam)^n\, (\det \mu)^n, \quad\text{if }X^*=X.
\end{align*}

If particular, if $n=2k$, for an odd integer $k$, then, \begin{alignat*}{2} \det \tau &= -1,
\qquad \text{if} \quad X^*=X^t,\\ &= 1, \qquad \text{if}\quad X^*=X . \end{alignat*}
\end{lemma}
\begin{corollary} \label{3.2} Let $n\geq 3$. Consider the homomorphism
\[\gam : \hat{D}\times \hat{D}
\longrightarrow (\GL(\End \fv_1))_{\perm}, \,\,\gamma (\lam ,\mu ) (v\otimes f)
=\lam v\otimes  (\mu^{-1})^*f,\] for $v\otimes f\in \fv_1\otimes \fv_1^*=\End \fv_1,$
where $(\mu^{-1})^*$ denotes the map induced by $\mu^{-1}$ on the dual space $\fv_1^*$.
 Then, $\gam$ induces
an embedding  of groups \[ \bar{\gamma}:(\hat{D}\times
\hat{D})/\Theta_n \hookrightarrow {(\GL(\End \fv_1))}_{\perm} ,\]
where $\Theta_n$ acts on $\hat{D}\times \hat{D}$ via: $z\cdot (\lam
,\mu ) = (z\lam ,z\mu )$, for $z\in \Theta_n$.

Moreover, $\Imo \bar{\gamma}$ contains the identity component of
${(\GL(\End \fv_1))}_{\perm}$.

Further, if  $n=2k$, for an odd integer $k$, then, $\bar{\gamma}$ is an isomorphism onto
${(\SL(\End \fv_1))}_{\perm}$.
\end{corollary}
Since the isotropy $\SL(\End \fv_1)_{\perm}$ is not contained in any proper
parabolic subgroup of $\SL(\End \fv_1)$,  Kempf's theorem
[Ke, Corollary 5.1] gives the following result observed in [MS1, Theorem 4.7]:
\begin{proposition} For $n \geq 3$,
$\SL(\End \fv_1)$-orbit of perm inside $S^n((\End \fv_1)^*)$ is closed.

In particular, an irreducible $\SL(\End \fv_1)$-module $M$ occurs in
$\bc[\overline{\GL(\End \fv_1)
\cdot
\perm}]$ if and only if $M^{\SL(\End \fv_1)_{\perm}}\neq 0$ (cf. the proof of Corollary
~\ref{1.6}).
\end{proposition}
 By exactly the same proof as that of Theorem ~\ref{2.6}, we get the following:

\begin{theorem} For $n\geq 3$, the subvariety  $\overline{\GL(\End \fv_1)
\cdot
\perm} \subset S^n((\End \fv_1)^*)$ is not normal.
\end{theorem}

We prove the following lemma for its application in the next section.

\begin{lemma} \label{3.5} Let $C=(c_{i,j}) \in \End \fv_1$ be such that
\[\perm (X+C)=\perm (X), \,\,\text{for all}\, X\in \End \fv_1.\]
 Then, $C=0$. \end{lemma}

\begin{proof} Take $X=(x_{i,j})$ with $x_{1,2} = \cdots  = x_{1,n} =0$. Then,
\begin{align}\label{e3.1} \perm (X) &=
\perm \begin{pmatrix} x_{1,1} &0 &\cdots &0 \\ x_{2,1} &x_{2,2} &
\cdots & x_{2,n}\\ \vdots &\vdots & &\vdots \\ x_{n,1} &x_{n,2}
&\cdots & x_{n,n} \end{pmatrix} \notag\\ &= x_{1,1} \perm X^{(1,1)}
,
\end{align}
where \[  X^{(1,1)} = \begin{pmatrix} x_{2,2} &\cdots &x_{2,n}\\ \vdots & &\vdots\\
x_{n,2}
&\cdots &x_{n,n} \end{pmatrix} . \]
By assumption, for any $X=(x_{i,j})$ as above,
\begin{align}  \label{e3.2}\perm (X) &= \perm (X+C) \notag\\
&= (x_{1,1}+c_{1,1}) \,
\perm \bigl( X^{(1,1)}+C^{(1,1)}\bigr) + c_{1,2}\, \perm \bigl( X^{(1,2)} +C^{(1,2)}\bigr)\notag\\
&\,\,\, +
\cdots + c_{1,n} \perm \bigl( X^{(1,n)} + C^{(1,n)}\bigr).
\end{align}
Now, $x_{1,1}$ divides the left side by ~\eqref{e3.1}, hence it must
also divide the right side of the above equation. Thus, \beqn
\label{e3.3} \sum^n_{j=1} c_{1,j} \perm \Bigl( X^{(1,j)} +
C^{(1,j)}\Bigr) = 0 \eeqn and (by equations
~\eqref{e3.1}-~\eqref{e3.3})
\[ \perm \Bigl( X^{(1,1)} + C^{(1,1)} \Bigr)= \perm \bigl(X^{(1,1)}\bigr) . \]
By induction, this gives \[ C^{(1,1)} \equiv 0 . \] By a similar
argument, \[ C^{(1,j)} = 0,\quad \,\text{for all}\, j . \]
Substituting  this in ~\eqref{e3.3}, we get
 \[
\sum^n_{j=1} c_{1,j} \perm X^{(1,j)} =0,\]
which gives  $c_{1,j} = 0$ for
all $ j$. Hence,
 \[C=0.\]
\end{proof}

\section{Functions on the orbit closure of $\fp$}

{\em We take in this and the subsequent sections} $3\leq n <m$.

 Recall the definition of the subspace $\fv_1\subset \fv$ from Section
 3. Let
 $\fv_1^{\bot}$ be the complementary subspace of $\fv$ with basis $\{e_{1}, \dots ,e_{m-n}\}$.
Consider the  function $\fp\in Q=S^m(E^*)$, defined by $\fp(X) =
x^{m-n}_{1,1} \perm (X^o),$ $X^o$ being the component of $X$ in the right down $n\times n$
corner $\left(\begin{smallmatrix} x_{1,1} &\quad &*\\ &\ddots \\ * && \underbrace{X^o}_n
\end{smallmatrix}\right)$, where any element of $\End \fv$ is represented by a $m\times m$-matrix
$X=(x_{i,j})_{1\leq i,j,\leq m}$ in the basis $\{e_i\}$.

Let $S$ be the subspace of $E$  spanned by
$e_{1,1}$ and $e_{i,j}$; $m-n+1 \leq i,j\leq m$; and $S^{\bot}$ be the complementary
subspace spanned by the set $\{e_{i,j}\}_{1\leq i,j,\leq m}\setminus
\{e_{1,1}, e_{i,j}\}_{m-n+1 \leq i,j\leq m}$ (where, as in Section 1,
$e_{i,j}:=e_i\otimes e_j^*$). Let $P$ be the maximal parabolic subgroup of
$G=\GL(E)$ which keeps the subspace $S^{\bot}$ of $E$ stable. Let $U_P$ be the
 unipotent radical of $P$ and let $L_P$ be the Levi subgroup of $P$ defined by:
 $L_P=L_P^1\times L_P^2$, where $L_P^1:=\Aut S^{\bot}$ and $L_P^2:=\Aut S$.

\begin{lemma} \label{4.1} The subgroups $L_P^1$ and $U_P$ act trivially on $\fp$.

Hence, $P\cdot \fp=L^2_P\cdot \fp$.
\end{lemma}
\begin{proof} Take $g\in U_{P}$. Then, since $U_{P}$ acts via identity on $S^{\bot}$, and
$g(X_2) \in X_2 + S^{\bot}\,\text{for all}\, X_2\in S$, we have (for any $X\in E$)
\begin{align*} (g^{-1}\fp)X &= \fp(g
X)\\ &= \fp(gX_1+gX_2),\qquad\text{where $X=X_1+X_2$, $X_1\in S^{\bot}$, $X_2\in S$}\\
&= \fp(X_1+Y_1+X_2),\, \quad\text{ for some $Y_1\in S^{\bot}$}\\ &= \fp(X_2)\\ &= \fp(X).
\end{align*} For $g\in L^1_P$, \begin{align*} (g^{-1}\fp)X &= \fp(gX)\\ &=
\fp(gX_1+gX_2), \qquad
\text{where}\, X=X_1+X_2, X_1\in S^{\bot},\, X_2\in S\\ &= \fp(gX_1+X_2)\\ &= \fp(X_2)\\ &= \fp(X) . \end{align*}

This proves the lemma. \end{proof}

Since $G/P$ is a projective variety,
\[\cy:=G\cdot(\overline{P\cdot\fp})=\overline{G\cdot\fp}\subset Q.\]
 Thus, we have a proper surjective morphism
 \[ \phi: {G}\times_{P} \bigl(\overline{P\cdot \fp}\bigr)= G\times_P
 \bigl(\overline{L_{P}^2\cdot \fp}\bigr)\twoheadrightarrow \cy, \,\,[g,x]
 \mapsto g\cdot x,\]
 for $g\in G$ and $x\in \overline{P\cdot\fp}.$ Let (for any $d\geq 0$)
\beqn \label{e4.1}
 \bc^d\bigl[ \overline{L_P^2\cdot
\fp}\bigr] = \bigoplus_{\lambda\in D(L_P^2)}\, n_{\lam}(d)\,
V_{L_P^2}(\lam )^* , \eeqn where $\bc^d[\overline{L_P^2\cdot \fp}]$
denotes the space of homogeneous degree $d$-functions with respect
to the embedding $\overline{L_P^2\cdot \fp}
 \subset Q$,
$D(L_P^2)$ denotes the set of dominant characters for the group
$L_P^2$ (with respect to its standard diagonal subgroup) consisting of
$\lam=(\lam_1\geq \dots \geq \lam_{n^2+1})$ with $\lam_i\in \bz$, and
$V_{L_P^2}(\lam )$ is the irreducible $L_P^2$-module with highest weight
$\lam$.

\begin{theorem} \label{4.2} For any $\lambda\in D(L_P^2)$ and $d\geq 0$ such that
$n_\lam(d)>0$, we have $\lam_1\leq 0$.

Moreover, as $G$-modules,
\[ \bc^d[\cy] = \bigoplus_{\lambda\in D(L_P^2)}\, n_{\lam}(d)\,
V_{G}(\hat{\lam})^*,\]
where $\hat{\lam}:=(0\geq \dots \geq 0\geq \lam_1\geq  \dots\geq
\lam_{n^2+1})
\in D(G)$ (with initial $m^2-n^2-1$ zeroes).

Further, the $G$-equivariant morphism $\phi$ induces an isomorphism of
 $G$-modules:
 \[\phi^*: \bc[\cy]\to \bc[G\times_{P} \bigl(\overline{P\cdot \fp}
 \bigr)].\]
\end{theorem}
\begin{proof}
 Observe
that, by Lemma ~\ref{4.1},  $\bc^d\bigl[ \overline{L_P^2\cdot \fp}\bigr]$
 is a $P$-module quotient
of $\bc^d[\overline{{G}\cdot \fp}]$ with
$U_P$ and $L_P^1$ acting trivially on $\bc^d\bigl[ \overline{L_P^2\cdot
\fp}\bigr]$. Thus, as $P$-modules,
\[\bc^d\bigl[ \overline{L_P^2\cdot \fp}\bigr]^*\simeq
\bigoplus_{\lambda\in D(L_P^2)}\, n_{\lam}(d)\,
V_{L_P^2}(\lam
)\hookrightarrow \bc^d[\cy]^*.\]
Take a nonzero $B_{L_P^2}$-eigenvector of weight $\lam$ in
$\bc^d\bigl[ \overline{L_P^2\cdot \fp}\bigr]^*$, where $B_{L_P^2}$ is the standard
Borel subgroup of $L_P^2$ consisting of upper triangular matrices. Then,
its image in $\bc^d[\cy]^*$ is a $B$-eigenvector of weight $\hat{\lam}$,
where $B$ is the standard Borel subgroup of $G$. In particular, for
any $\lam \in D(L_P^2)$ such that $n_\lam(d)>0, \hat{\lam}\in D(G)$
(since $\bc^d[\cy]^*$ is a $G$-module). Hence, $\lam_1\leq 0$ and
$\bigoplus_{\lambda\in D(L_P^2)}\, n_{\lam}(d)\,
V_{G}(\hat{\lam})\subset \bc^d[\cy]^*$. Dualizing, we get the $G$-module surjection:
\beqn \label{e4.2} \bc^d[\cy]\twoheadrightarrow \bigoplus_{\lambda\in D(L_P^2)}\, n_{\lam}(d)\,
V_{G}(\hat{\lam})^*.\eeqn
From the surjection $\phi$, we obtain  the $G$-module injective map:
\begin{align*}\phi^*: \bc^d[\cy] \hookrightarrow
&H^0(G/P, \bc^d\bigl[ \overline{L_P^2\cdot \fp}\bigr])\\
&= \bigoplus_{\lambda\in D(L_P^2)}\, n_{\lam}(d)\,
H^0(G/P, V_{L_P^2}(\lam)^*), \,\,\text{where $U_P$} \\
&\qquad \text{and $L_P^1$ act
 trivially on $V_{L_P^2}(\lam)^*$}\\
 &\simeq \bigoplus_{\lambda\in D(L_P^2)}\, n_{\lam}(d)\,
V_{G}(\hat{\lam})^*,
\end{align*}
where the last isomorphism follows from [Ku1, Lemma 8]. Combining the
injection
$\phi^*$ with \eqref{e4.2}, we get that $\phi^*$ is an isomorphism,
proving the theorem.
\end{proof}

\begin{proposition} \label{4.8} The isotropy of $\fp$ under the group $P$ is the same as that under the
group $G$. \end{proposition}

\begin{proof} First of all $G/P = W'_PU^-_P\, P/P,$ where $U^-_P$
is the opposite  of the unipotent radical $U_P$
of $P$ and $W'_P$ is the set of all the  smallest coset
representatives  of $W/W_P$. (This follows since the right side is an open subset of
$G/P$ which is $T$-stable and contains all the $T$-fixed points of
$G/P$.)

Take $w\in W'_P$, $u\in U^-_P$, $r\in L^2_P$ such that $wur\cdot \fp
= \fp$. Then,
 \beqn \label{e4.3}
\fp(r^{-1}u^{-1}w^{-1}X) = \fp(X),\quad\text{ for any } \,X=X_1+X_2\in
E= S^{\bot}\oplus S. \eeqn
In particular, for  $X=wX_2$, we get
\beqn \label{e4.4} \fp(r^{-1}u^{-1}X_2) =
\fp(wX_2). \eeqn
We have $ u^{-1}X_2 = X_2,$ thus
 \beqn\label{e4.5} \fp(r^{-1}u^{-1}X_2) = \fp(r^{-1}X_2). \eeqn

Well order a basis of $S$ as $v_1,v_2,\dots ,v_d$ ($d=n^2+1$) and
also a basis $v_{d+1},\dots ,v_{m^2}$ of $S^{\bot}$. Then, $w$ can
be represented as the permutation $i\mapsto n_i$ with
\[ n_1 < \cdots < n_d, n_{d+1} <
\cdots < n_{m^2} . \]
 For $ X_2 = \sum^d_{i=1} z_iv_i\in S$,
\beqn \label{e4.5}
\fp(wX_2) = \fp\Bigl(
\sum^d_{i=1} z_i v_{n_i}\Bigr) = \fp\Bigl( \sum_{i\leq i_o} z_i
v_{n_i}\Bigr),
\eeqn
where $1\leq i_o\leq d$ is the maximum integer such that $n_{i_o} \leq
d$. In particular, $\fp(wX_2)$ only depends upon the variables $z_1,\cdots
,z_{i_o}$. Thus, by the identities \eqref{e4.4} -- \eqref{e4.5},
\[\fp\Bigl(
r^{-1} \sum^d_{i=1} z_i v_i\Bigr) = \fp\Bigl( \sum_{i\leq i_o} z_i
v_{n_i}\Bigr), \qquad \,\text{for any}\, z_i\in\bc,\]
 which gives
\[ \fp\Bigl( r^{-1} \sum^d_{i=1} z_iv_i\Bigr) =
\fp\Bigl( r^{-1}\Bigl( \sum^d_{i=1} z_iv_i + \sum_{d\geq j>i_o}
b_jv_j\Bigr)\Bigr), \quad\text{for any }b_j\in\bc.\] Thus,
\[\fp\Bigl(\sum^d_{i=1} z_iv_i\Bigr) =
\fp\Bigl(\sum^d_{i=1} z_iv_i + r^{-1}\sum_{d\geq j>i_o}
b_jv_j\Bigr). \]
 Applying Lemma ~\ref{3.5}, it is easy to see that
 $ \sum_{d\geq j> i_o} b_jv_j =0 $ (for any $b_j\in \bc$).  Thus, $i_o =d$, i.e.,
$w=1$.

Taking $X=X_2\in S$ in ~\eqref{e4.3}, we get (since $w=1$)
$\fp(r^{-1}X_2) = \fp(X_2)$, which is equivalent to $\fp(r^{-1}X) =
\fp(X)$, for all $ X\in E$.
 Thus, $r$ is in the isotropy of $\fp$ and hence $u$ is in the
isotropy of $\fp$, i.e., $ \fp(u^{-1}X) = \fp(X),$  for all $X=X_1+X_2\in
E$. This gives $\fp(X_1+X_2+Y_2) = \fp(X_1+X_2),$  where $Y_2 :=
u^{-1}X_1-X_1\in S$. Hence, $ \fp(X_2+Y_2) = \fp(X_2)$, for all $
X_2\in S$  and any $Y_2$ of the form $u^{-1}X_1-X_1$, for some
$X_1\in S^\bot$. Applying Lemma ~\ref{3.5} again, we see that
 $ Y_2 =0$, hence  $ u_{|_{S^{\bot}}} = \text{Id.}$ Thus,
$ u=1 .$  This proves the proposition since $U_P$ and $L^1_P$
stabilize $\fp$. \end{proof}

\begin{corollary} \label{4.4} The restriction $\phi_o$ of the map  $\phi$ to
$ {G}\times_{P} (P\cdot \fp )$ is a biregular isomorphism onto
  $G\cdot \fp$.

  Morepver,  $\phi^{-1}(G\cdot \fp)= {G}\times_{P} (P\cdot \fp ).$
\end{corollary}

\begin{proof} Of course, $\phi_o$ is surjective. We next claim that $\phi_o$ is injective. Take
$\phi_o [g,  \fp] = \phi_o [g_1, \fp]$, i.e., $g\cdot
\fp = g_1\cdot \fp,$ which is equivalent to $ \bigl(
g_1^{-1}g\bigr)\cdot \fp = \fp$, i.e., $g_1^{-1}g
\in G_\fp = P_\fp,$ by Proposition ~\ref{4.8}. Thus, $ g_1^{-1}g =
\tilde{r},$   for some  $\tilde{r}\in P_\fp\subset P.$  Hence,
$[g,  \fp] =  [g_1, \fp]$, proving  that $\phi_o$ is bijective. Since $G\times_P (P\cdot \fp)$
and $G\cdot \fp$ are both smooth, $\phi_o$ is an isomorphism
(cf. [Ku2, Theorem A.11]).

To prove that $\phi^{-1}(G\cdot \fp)= {G}\times_{P} (P\cdot \fp ),$
 take $[g,y]\in
G\times_P\,(\overline{P\cdot \fp})$ such that $\phi[g,y] \in
G\cdot \fp,$  i.e.,  $g\cdot y = h\cdot \fp,$ for some  $ h\in G$.
This gives $y \in G\cdot
\fp\cap \overline{P\cdot \fp} .$
 But, $P\cdot \fp$ is closed in $G\cdot \fp$ by the first part of the
 corollary and hence  $y\in P\cdot \fp$, establishing the claim.
\end{proof}

Let ${S_1}$ be the subspace of $S$ spanned by  $e_{i,j}$, $m-n+1\leq
i,j\leq m$. Consider the maximal parablic subgroup $R$ of $L_P^2=\Aut S$,
consisting of those $g\in \Aut S$ which stabilize the line
  $\Bbb
C e_{1,1}.$  Then, $L_R:=\Aut (\bc e_{1,1})\times \Aut S_1$ is a Levi
subgroup of $R$. Let $U_R$ be the unipotent radical of $R$ and $U_R^-$ the
opposite unipotent radical.
\begin{proposition} \label{4.5} The isotropy of $\fp$ under the
group $L_P^2$
 is the same as the
isotropy of the Levi subgroup $L_R$.
\end{proposition}

\begin{proof}
In the proof, we let $i,j$ run over $m-n+1\leq i,j\leq m$. Any element
$u\in U_R$ is given by:
$ue_{1,1}=e_{1,1}$, $u\, e_{i,j}=e_{i,j} + a_{i,j}e_{1,1}$, for some
$a_{i,j}\in\bc$. Similarly, $U_R^-$ consists of $u^-$ such that
$u^-e_{i,j}=e_{i,j}$ and $u^-e_{1,1} = e_{1,1} + \sum
c_{i,j}e_{i,j}$. Any element of $L_P^2$ can be written as $wu^-ug$
 (for some
$g\in L_R, u\in U_R, u^-\in U_R^-$ and $w$ either the identity element
or a $2$-cycle $((1,1),(i,j))$). Take any $X= x_{1,1}e_{1,1} +
\sum x_{i,j}e_{i,j}\in S$. By $X_{S_1}$ we mean $\sum x_{i,j}e_{i,j}$ and by
$(X)_{1,1}$ we mean $x_{1,1}$.
 \begin{align*} ((wu^-ug)^{-1} \cdot\fp)(X)
&= \fp(wu^-ug\, X)\\
&= \Bigl( (wu^-ug\, X)_{1,1}\Bigr)^{m-n} \perm \Bigl( (wu^-ug\, X)_{S_1}
\Bigr).
 \end{align*} So, if $(wu^-ug)^{-1} \in (L_P^2)_\fp$, then
 \[ \bigl( (wu^-ug)^{-1}\cdot \fp\bigr) (X) =
\fp(X)= x_{1,1}^{m-n} \perm (X_{S_1}), \quad\text{for all $X\in S$}.
\]
Since no linear form divides $\perm$, we get
\beqn \label{e4.6} \al x_{1,1} = (wu^-ugX)_{1,1},
\qquad\text{for some constant }\al\neq 0\in\bc, \,\text{and}
\eeqn
\begin{align}\label{e4.7} \beta\perm (X_{S_1}) &= \perm \bigl( (wu^-ug\, X)_{S_1}\bigr),
\quad\text{for some constant}\,\beta\neq 0\in\bc \notag\\
&= \perm \bigl(
(wu^-ug\, X_{S_1} + x_{1,1} wu^-ug\, e_{1,1})_{S_1}\bigr).
\end{align} Since the left hand side of \eqref{e4.7} is independent of $x_{1,1}$, we
get \[ \perm\bigl( (wu^-ug\, X)_{S_1}\bigr) = \perm\bigl( (wu^-ug\, X)_{S_1}
+ (\alpha_{1,1}wu^-ug\, e_{1,1})_{S_1}\bigr), \] for all $X\in S$ and $\alpha_{1,1}\in \bc$.

Since $wu^-ug\in \Aut S$, as $X$ varies over $S$,
 $(wu^-ug\, X)_{S_1}$ varies over all of ${S_1}$. Thus, by
Lemma ~\ref{3.5}, \begin{equation}\label{e4.8} (wu^-ug\, e_{1,1})_{S_1} = 0.
\end{equation}
Now,
\begin{align} \label{e4.9} u^-u ge_{1,1} &= u^- (\lam\, e_{1,1}), \qquad\text{for
some }\lam\neq 0\notag\\ &= \lam \bigl( e_{1,1} + \sum c_{i,j}
e_{i,j}\bigr). \end{align}
Thus,
If $w$ is the $2$-cycle $((1,1), (i_o,j_o))$ for some  $m-n+1\leq i_o,j_o
\leq m$,  then\[ wu^- ug\, e_{1,1} = \lam
\Bigl( e_{i_o,j_o} + \sum_{(i,j)\neq (i_o,j_o)} c_{i,j} e_{i,j} +
c_{i_o,j_o}e_{1,1}\Bigr) . \] In particular, $(wu^-ug\, e_{1,1})_{S_1}
\neq 0$, a contradiction to the identity \eqref{e4.8}. Thus, $w=1$. By the equations \eqref{e4.8}--
\eqref{e4.9}, we
get \[ c_{i,j} =0 \qquad\text{for all }i,j . \] Thus, $u^- =1$.

By equation \eqref{e4.6}, we get \[ \al\,
x_{1,1}=(wu^-ug\, X)_{1,1} = (ug\,
X)_{1,1} = \bigl( ug(X_{S_1}+x_{1,1}\, e_{1,1})\bigr)_{1,1}.\]
  In particular, $(ug\, X_{S_1})_{1,1} = 0$. Since
$g$ maps $S_1$ onto ${S_1}$, we get
\[ (u\, e_{i,j})_{1,1} = 0,
\qquad \text{for all}\,\, m-n+1\leq i,j\leq m.\]
 Hence,  $a_{i,j} = 0 .$ Thus,
$u=1$ as well. This proves the proposition. \end{proof}

\begin{corollary} Let $3\leq n <m$. Then, each irreducible component of $\overline{L_P^2 \cdot\fp}
\backslash L_P^2\cdot \fp$ is of codimension $1$ in
$\overline{L_P^2 \cdot\fp}$. \end{corollary}

\begin{proof} By the last proposition, the isotropy of $\fp$ inside $L^2_P$ is
the same as that of the isotropy of $\fp$ inside $L_R$.
 For any $\lam \in \bc^*$,  take $\tau_\lam\in \Aut (\bc e_{1,1})$ defined by
 $e_{1,1} \mapsto
\lam\, e_{1,1} $.  Then, for any $g\in \Aut S_1$ and $X=x_{1,1}e_{1,1}+
X_1$ with $X_1\in S_1$, we have
\begin{align}\label{e4.10}
\Bigl((\tau_\lam, g)\cdot \fp\Bigr)(X) &= \fp\Bigl( \lam^{-1}x_{1,1}e_{1,1}
+ g^{-1}X_1\Bigr)\notag\\ &= (\lam^{-1}x_{1,1})^{m-n}\, \perm (g^{-1}X_{1}).
\end{align}
Thus, $(\tau_\lam, g)\in (L_R)_\fp$ if and only if
$(\lam^{\frac{1}{n}})^{m-n}g\in (\Aut S_1)_{\perm}$, for some
$n$-th root $\lam^{\frac{1}{n}}$ of $\lam$. Considering the projection
to the first factor $(L_R)_\fp\to \Aut (\bc e_{1,1})=\bc^*$ and using
Corollary ~\ref{3.2}, it is easy to see that $(L_R)_\fp=(L_P^2)_\fp$ is
reductive. Thus, $L^2_P\cdot \fp$ is an affine variety.
 Of course,
$\overline{L^2_P\cdot \fp}$ is an affine variety. Moreover,
$0\in (\overline{L^2_P\cdot \fp})\backslash L^2_P\cdot \fp$ by \eqref{e4.10}.
 Thus,
 $(\overline{L^2_P\cdot
\fp})\backslash L_P^2\cdot \fp$ is nonempty and each of its irreducible
components is of codimension $1$ in $\overline{L^2_P\cdot
\fp}$ by the following lemma. \end{proof}

We recall the following well known result from algebraic geometry. For the lack of
reference, we include a proof.
\begin{lemma} \label{4.6a} Let $X$ be an irreducible affine variety and let $X^o\subset X$
be an open normal affine subvariety. Then, each irreducible component of
$X\setminus X^o$ is of codimension $1$ in $X$.
\end{lemma}
\begin{proof} Let $\pi:\tilde{X}\to X$ be the normalization of $X$. Then,
$X^o$ being normal and open subvariety of $X$, $\pi:\pi^{-1}(X^o)\to X^o$
is an isomorphism. We identify $\pi^{-1}(X^o)$ with  $X^o$ under $\pi$. Decompose
$\tilde{X}\setminus X^o=C_1\cup C_2$, where $C_1$ (resp. $C_2$) is the union of
codimension $1$ (resp. $\geq 2$) irreducible components of $\tilde{X}
\setminus X^o$. Then, by Hartog's theorem, the inclusion $i: X^o\subset
\tilde{X}\setminus C_1$ induces an isomorphism $i^*:\bc[\tilde{X}
\setminus C_1]\simeq \bc[X^o]$ of the rings of regular functions. Let $f$ be the
inverse of $i^*$. Then, $X^o$ being affine, there exists a morphism
$j: \tilde{X}\setminus C_1 \to X^o$ such that the induced map $j^*=f$
and $j_{|X^o}=\text{Id}$ (cf. [H, Proposition 3.5, Chap. I]). Since the composite
morphism $i\circ j: \tilde{X}\setminus C_1 \to \tilde{X}\setminus C_1$ restricts
to the identity map on $X^o$ and $X^o$ is dense in $\tilde{X}\setminus C_1$,
$i\circ j=\text{Id}$. In particular, $i$ is surjective, i.e., $X^o=
\tilde{X}\setminus C_1$. Thus,
\[X\setminus X^o=\pi(\tilde{X}\setminus X^o)=\pi(C_1).\]
But, since $\pi$ is a finite morphism, $\pi(C_1)$ is closed in $X$ and, moreover, all the
 irreducible components of $\pi(C_1)$ are of codimension $1$ in $X$.
 \end{proof}

As another corollary of Proposition ~\ref{4.5} (together with Corollary ~\ref{3.2},
Lemma ~\ref{4.1}, Proposition ~\ref{4.8} and identity ~\eqref{e4.10}),
we get the following.
\begin{corollary} For $3\leq n <m$, $\dim \cy=m^2(n^2+1)-2n+1.$
\end{corollary}
\section{A partial desingularization of $\overline{L_P^2\cdot \fp}$}

By virtue of the results in the last section (specifically Theorem
~\ref{4.2}), study of the $G$-module $\bc[\cy]$ reduces to that of the
$L_P^2$-module $\bc[\overline{L_P^2\cdot \fp}]$.

\begin{definition} \label{5.1} {\em Define the morphism \[ \beta : L^2_P \times_R
(\overline{R\cdot
\fp}) \to \overline{L^2_P\cdot \fp}, \,\,[g, f] \mapsto g\cdot
f , \]
for $g\in L_P^2, f\in \overline{R\cdot \fp}$, where the closure
$\overline{R\cdot \fp}$ is taken inside $S^m(E^*)$.

Since $L_P^2/R$ is a projective variety, $\beta$ is a proper and
surjective morphism.}
 \end{definition}

\begin{lemma} \label{5.2} The inverse image $\beta^{-1} (L^2_P\cdot \fp )
=L^2_P \times_R (R\cdot \fp).$ Moreover, the restriction
$\beta_o$ of $\beta$ to $L^2_P \times_R (R\cdot \fp)$ is a
biregular isomorphism onto $L_P^2\cdot \fp$.  \end{lemma}

\begin{proof} Take $[g,f]\in \beta^{-1}(L_P^2\cdot \fp).$ Then,
$f\in (L_P^2\cdot \fp)\cap \overline{R\cdot \fp}.$ But, $L_P^2/R$
being projective, $R\cdot \fp$ is closed in $L_P^2\cdot \fp.$ Thus,
$(L_P^2\cdot \fp)\cap \overline{R\cdot \fp}=R\cdot \fp.$ This proves
the first part of the lemma.

 By Proposition ~\ref{4.5}, the isotropy of
$\fp$ inside $L^2_P$ is the same as that in $R$. From this the
injectivity of $\beta_o$ follows easily. Since $\beta_o$ is a bijective
morphism
between smooth varieties, it is a biregular isomorphism.
 \end{proof}

As in Section 3, consider $\perm\in S^n(S_1^*)$, where $S_1$ is
viewed as $\End \fv_1$ and $\fv_1$ is equipped with the basis
$\{e_{m-n+1}, \dots, e_m \}.$ Moreover, the decomposition
$E=S^{\bot}\oplus \bc e_{1,1}\oplus S_1$ gives rise to the
projection $E\to S_1$ and, in turn, an embedding
$S^n(S_1^*)\hookrightarrow S^n(E^*)$. Thus, we can think of $\perm
\in S^n(E^*).$ Let
\[\cz^o:=(\Aut S_1)\cdot \perm \subset S^n(E^*),\]
where $\Aut S_1$ is to be thought of as the subgroup of $G$ by
extending any automorphism of $S_1$ to that of $E$ by defining it to
be the identity map on $S^{\bot}\oplus \bc e_{1,1}$. Let $\cz$ be the
closure of $\cz^o$ in $S^n(E^*)$.

Consider the standard (dual) action of $L_P^2=\Aut S$ on $S^*$.
In particular, we get an action of $R$ on $S^*$.  Also, it is easy
to see that $U_R$ and $\Aut (\bc e_{11})$ act trivially on $\cz^o$
(and hence on $\cz$) under the standard action of $G$ on $S^n(E^*)$.
 In particular, $\cz$ is a $R$-stable closed subset of $S^n(E^*)$
 (under the standard action of $R$).

Consider the morphism
  \[
\bar{\al} : S^*\times \cz \to Q, \,\,
(\lam, f) \mapsto \bar{\lam}^{m-n}f, \]
for $\lam\in S^*$ and $f\in \cz$,
where $\bar{\lam}\in E^*$ is the image of $\lam$ under the inclusion
$S^*\hookrightarrow E^*$ induced from the projection $E\to S$.
  Then, $\bar{\alpha}$ is $R$-equivariant under the diagonal action of
$R$ on $S^*\times \cz$.
  Define an action of $\bc^*$ on $S^*\times \cz$ via
  \beqn \label{e5.1}
z(\lam ,f) = (z\lam , z^{n-m}f).
  \eeqn
This action commutes with the action of $R$.  Then, $\bar{\al}$ clearly
factors through the $\bc^*$-orbits, and hence we get an $R$-equivariant
morphism
  \[  \al: (S^*\times \cz)//\bc^* \to Q .  \]

  \begin{proposition} \label{5.3} The above morphism $\al$ is a finite morphism
  with image precisely equal to $\overline{R\cdot\fp}$.

Moreover, $\al^{-1}(R\cdot\fp )=\bigl( (S^*\backslash S^*_1)\times
 \cz^o\bigr) //\bc^*$ and the map $\al_o$ obtained from the
 restriction of $\al$ to $\bigl( (S^*\backslash S^*_1)\times \cz^o\bigr)
 //\bc^*$ is a biregular isomorphism
  \[
\al_o: \bigl( (S^*\backslash S^*_1)\times \cz^o\bigr)
//\bc^* \simto R\cdot\fp ,
  \]
where $S^*_1$ is thought of as a subspace of $S^*$ via the projection
$S =\bc e_{1,1}\oplus S_1 \to S_1$.

In particular, $\al$ is a proper and birational morphism onto
$\overline{R\cdot\fp}$.
  \end{proposition}

  \begin{proof}  Consider the $\bc^*$-equivariant closed embedding
  \[  S^*\times \cz \hookrightarrow E^*\times S^n(E^*),  \]
where $\bc^*$ acts on the right side by the same formula as ~\eqref{e5.1}.  This gives rise to the closed embedding
    \[
\iota : S^*\times \cz//\bc^*\hookrightarrow E^*\times S^n(E^*)//\bc^* .
   \]
We next claim that the morphism
  \[  \psi : E^*\times S^n(E^*)//
  \bc^* \to Q = S^m(E^*),
   \]
induced from the map
$(\bar{\lam}, f) \mapsto \bar{\lam}^{m-n}f, $
for $\bar{\lam}\in E^*$ and $f\in S^n(E^*)$, is a finite morphism.
 Define a new $\bc^*$ action on $E^*\times S^n(E^*)$ by
  \[
t\odot (\bar{\lam},f) = (t\bar{\lam}, tf), \quad\text{ for }
t\in \bc^*.
  \]
This $\bc^*$-action commutes with the $\bc^*$-action given by
~\eqref{e5.1}.
Thus, we get a $\bc^*$-action (still denoted by $\odot$) on
$E^*\times S^n(E^*)//\bc^*$.  Also, define a $\bc^*$-action on
 $S^m(E^*)$ by
  \[
t\odot f = t^{m-n+1}f, \quad\text{ for $t\in\bc^*$ and $f\in S^m(E^*)$}.
  \]
Then, $\psi$ is $\bc^*$-equivariant.  Moreover, $\psi^{-1}(0) =
(0\times S^n(E^*) \cup E^*\times 0)//\bc^* = \{ 0\}$.  Thus, by
Lemma ~\ref{2.2} (applied to the map $\psi$ considered as a map:
$E^*\times S^n(E^*)//\bc^* \to \overline{\text{Im} \psi}$), $\psi$ is a finite morphism.

Since $\al =\psi\circ\iota$, we get that $\al$ is a finite morphism.

We next calculate $\al^{-1}(\fp )$.  Let $[\lam ,f]\in \al^{-1}(\fp )$,
where $[\lam ,f]$ denotes the image of $(\lam ,f)$ in
$S^*\times \cz//\bc^*$.  Then,
  \beqn \label{e5.2}
\bar{\lam}^{m-n}f = \fp = \bar{\lam}^{m-n}_o\perm ,
  \eeqn
where $\lam_o\in S^*$ is defined by $\lam_o (ze_{1,1}+X_1)=z$, for
any $z\in\bc$ and $X_1\in S_1$.

Since $\bar{\lam}$ does not divide $\perm$, from ~\eqref{e5.2} we get
  \[
\lam =a\lam_o \text{  and  } f=a^{n-m}\perm,\,\,\text{for some}\,
a\in\bc^*,
  \]
 which gives
  \[
[\lam ,f] = [\lam_o, \perm ].
  \]
Thus, $\al^{-1}(\fp )$ is a singleton and hence so is $\al^{-1}
(r\cdot\fp )$ for any $r\in R$ (by the $R$-equivariance of $\al$).
 In particular,
  \begin{align*}
\al^{-1}(R\cdot\fp ) &= R\cdot [\lam_o, \perm ]\\
&= \bigl(\Aut (\bc e_{1,1})\, U_R\Aut (S_1)\bigr) \cdot [\lam_o ,\perm ]\\
&= \bigl(\Aut (\bc e_{1,1})\, U_R\bigr) \cdot [\lam_o ,\cz^o ],\,\,
 \text{ since}\, \Aut (S_1)\cdot \lam_o=\lam_o\\
&= \bigl[(\Aut (\bc e_{1,1})\, U_R) \cdot \lam_o ,\cz^o \bigr],\,\,
 \text{ since}\, \Aut (\bc e_{1,1})\\
 &\quad \text{and}\, U_R \,\,\text{act trivially on} \, \cz^o\\
&= [S^*\backslash S^*_1, \cz^o]\\
&= \bigl( (S^*\backslash S^*_1)\times \cz^o\bigr)//\bc^*.
  \end{align*}
  Observe that all the $\bc^*$-orbits in $(S^*\backslash S^*_1)\times
  \cz^o$ are closed in $S^*\times \cz$ and hence
  $\bigl( (S^*\backslash S^*_1)\times \cz^o\bigr)//\bc^*=
  \bigl( (S^*\backslash S^*_1)\times \cz^o\bigr)/\bc^*$ can be thought of as
  an open subset of $\bigl( S^*\times \cz\bigr)//\bc^*.$
 This proves that ${\al_o}$ is a bijective morphism between smooth irreducible
 varieties and hence it is a biregular isomorphism (cf. [Ku2, Theorem A.11]).

Finally, since $\al$ is a finite morphism (in particular, a
proper morphism), Im $\al$ is closed in $Q$ and contains $R\cdot\fp$.
 Thus, Im $\al \supset \overline{R\cdot\fp}$.  But, since
 $\bigl( (S^*\backslash S^*_1)\times \cz^o\bigr)//\bc^*$ is dense
 in $S^*\times \cz//\bc^*$, we get Im $\al\subset \overline{R\cdot\fp}$
 and hence Im $\al = \overline{R\cdot\fp}$.

This completes the proof of the proposition.
  \end{proof}
\begin{remark} {\em Even though we do not need, the above map $\al$ is a
bijection onto its image.}
\end{remark}
Combining Lemma ~\ref{5.2} with Proposition ~\ref{5.3}, we get the
following:

\begin{corollary} We have \begin{align*} \bc &\bigl[ \overline{L^2_P\cdot
\fp}] \overset{\beta^*}\hookrightarrow \bc \bigl[ L^2_P \times_R
(\overline{R\cdot \fp})\bigr]
\simeq H^0 \bigl( L^2_P/R, \bc [ \overline{R\cdot \fp}]\bigr)\\
 &\overset{\alpha^*}\hookrightarrow H^0 \bigl( L^2_P/R, \bc [S^*\times
\cz]^{\bc^*}\bigr).
\end{align*}
 \end{corollary}

\section{Determination of $H^0\bigl(L_P^2/R, \bc[S^*\times
\cz]^{\bc^*}\bigr)$}

We continue to follow the notation from the last section. In particular,
$3 \leq n <m$. For any
$d\geq 0$, we have the canonical inclusion:
  \[
j: H^0 \Bigl( L_P^2/R, (\bc [S^*] \otimes \bc^d[\cz])^{\bc^*}\Bigr)
\hookrightarrow H^0\Bigl( L_P^2/R, (\bc [S^*\backslash S_1^* ]
\otimes \bc^d[\cz])^{\bc^*}\Bigr),\] where $\bc^d[\cz]$ denotes the
space of degree $d$-homogeneous functions on $\cz\subset S^n(E^*)$.
Thus, $\bc^d[\cz]$ is a quotient
 of $S^d(S^n(E))$. In this section, we will
determine the image of $j$.

For any $R$-module $M$, $H^0(L_P^2/R, M)$ can canonically be
identified with the space of regular maps
\[
\Bigl\{ \phi : L_P^2 \to M:
 \,\, \phi (\ell r) = r^{-1}\cdot (\phi (\ell )),\, \forall \ell\in L_P^2, r\in
 R\Bigr\}.
 \]
 Thus, by the Peter-Weyl theorem and the Tannaka-Krein duality (cf.
 [BD, Chap. III]),
 \beqn \label{e6.1}
H^0\Bigl( L_P^2/R, M\Bigr) \simeq \bigoplus_{\lam=(\lam_1\geq \cdots
\geq\lam_{n^2+1})\in D(L_P^2)} V_{L_P^2}(\lam)^*\otimes \Hom_R
\Bigl( V_{L_P^2}(\lam)^*, M\Bigr).\eeqn We will apply this to the
cases $M=(\bc [S^*] \otimes \bc^d[\cz])^{\bc^*}$ and $M=(\bc
[S^*\backslash S_1^* ] \otimes \bc^d[\cz])^{\bc^*}$.

\begin{lemma} \label{6.1} Take any $\lam=(\lam_1\geq \cdots
\geq\lam_{n^2+1})\in D(L_P^2)$ and any $d\geq 0$. Then, the
canonical inclusion
\[\Hom_R \bigl( V_{L_P^2}(\lam)^*, (\bc
[S^* ]\otimes \bc^d[\cz])^{\bc^*}\bigr) \hookrightarrow \Hom_R \bigl(
V_{L_P^2}(\lam)^*, (\bc [S^*\backslash S_1^* ]\otimes
\bc^d[\cz])^{\bc^*}\bigr)\] is an isomorphism if
$\lam_1\leq 0$.

Moreover, if $\lam_1 >0$, then the left side is $0$.
\end{lemma}
\begin{proof}
Take $\phi\in\Hom_R \bigl( V_{L_P^2}(\lam)^*, (\bc [S^*\backslash
S_1^* ]\otimes \bc^d[\cz])^{\bc^*}\bigr)$.  Let $v^*_{\lam}\in
V_{L_P^2}(\lam)^*$ be the lowest weight vector of weight $-\lam$.
Then, $\phi$ is completely determined by its value on $v^*_{\lam}$.
Let
  \[
\phi_1 := \phi (v^*_{\lam}) : (S^*\backslash S_1^*) \times \cz \to \bc
  \]
be the corresponding map. For $z\in \bc^*$, take the diagonal matrix
$\hat{z} = [z, 1, \dots, 1] \in L_P^2$ with respect to the basis
$\{e_{1,1}, e_{i,j}\}_{m-n+1\leq i,j\leq m}$. Then, $\phi (\hat{z}
v^*_{\lam}) = \hat{z}\cdot \phi (v^*_{\lam} )$, i.e.,
$e^{-\lam}(\hat{z})\phi_1 = \hat{z}\cdot \phi_1.$ This gives
$z^{-\lam_1}\phi_1 = \hat{z}\cdot \phi_1$, i.e.,
 \begin{align}\label{e6.2} z^{-\lam_1}\phi_1\bigl( (z_{1,1},
z_{i,j}), x\bigr) &=
  \phi_1\bigl( \hat{z}^{-1} \bigl( (z_{1,1},z_{i,j}), x\bigr)\bigr) \notag \\
  &= \phi_1 \bigl( (z z_{1,1},z_{i,j}), x\bigr) ,
  \end{align}
  where $\{z_{1,1}, z_{i,j}\}$ are the coordinates on $S^*$ with
  respect to the basis $\{e_{1,1}, e_{i,j}\}$ of $S$.
Write
  \[
\phi_1\bigl( (z_{1,1}, z_{i,j}), x\bigr) = \sum_{\ell\in\bz}
z^{\ell}_{1,1}\, P_{\ell} (z_{i,j}, x),
  \]
for some $P_{\ell} (z_{i,j}, x) \in\bc [S_1^*]\otimes \bc^d [\cz]$.
Equation ~\eqref{e6.2} gives
  \[
z^{-\lam_1} \sum_{\ell\in\bz} z^{\ell}_{1,1} P_{\ell} (z_{i,j}, x) =
\sum_{\ell\in\bz} z^{\ell} z^{\ell}_{1,1} P_{\ell} (z_{i,j}, x),
  \]
for all $z_{1,1}, z\in \bc^*$, $z_{i,j}\in\bc$ and  $x\in \cz$.  For
any $\ell\in\bz$ such that $P_{\ell} (z_{i,j}, x)\neq 0$ (for some
$z_{i,j}\in\bc$ and some $x\in \cz$), from the above equation, we get
$z^{-\lam_1} = z^{\ell}$.  In particular,
  \[
\phi_1 \bigl( (z_{1,1}, z_{i,j}), x\bigr) = z^{-\lam_1}_{1,1}
P_{-\lam_1} (z_{i,j}, x).
  \]
Thus, if nonzero,  $\phi_1 : (S^*\backslash S_1^*) \times \cz \to\bc$
extends to a morphism $S^*\times \cz\to \bc$ iff $-\lam_1\geq 0$. This proves the
lemma.
\end{proof}
As a corollary of the above lemma and the identity ~\eqref{e6.1}, we get
 the following.

  \begin{proposition} \label{6.2} For any $d\geq 0$, let
  \[
H^0\Bigl( L_P^2/R, (\bc [S^*\backslash  S_1^* ] \otimes
\bc^d[\cz])^{\bc^*}\Bigr) = \bigoplus_{\lam =(\lam_1\geq \cdots\geq
\lam_{n^2+1})\in D(L_P^2)} m_{\lam}(d)\, V_{L_P^2}(\lam)^*.
  \]
 Then,
 \[
H^0\Bigl( L_P^2/R, (\bc [S^*] \otimes \bc^d[\cz])^{\bc^*}\Bigr) =
\bigoplus_{\lam =(\lam_1\geq \cdots\geq \lam_{n^2+1})\in D(L_P^2):
\lam_1\leq 0} m_{\lam}(d)\, V_{L_P^2}(\lam)^* .
  \]
  \end{proposition}

Define a new action of $R$ on $\cz$ by \beqn \label{e6.0} r\odot
x=\chi(r)^{n-m} r\cdot x,\eeqn where $\chi:R\to \bc^*$ is the
character defined by $\chi(r)=(re_{1,1})_{1,1},$ where $(X)_{1,1}$
is defined in the proof of Proposition ~\ref{4.5}.

\begin{lemma} \label{6.3} For any $d\geq 0$, there is a canonical
isomorphism of $L_P^2$-modules:
\[H^0\Bigl( L_P^2/R, (\bc [S^*\backslash  S_1^* ] \otimes
\bc^d[\cz])^{\bc^*}\Bigr)\simeq H^0\Bigl( L_P^2/L_R,
\bc^d[\cz]^{\chi}\Bigr),\] where $\bc^d[\cz]^{\chi}$ is the same space
as $\bc^d[\cz]$ but the $L_R$-module structure on $\bc^d[\cz]^{\chi}$ is
induced from the action $\odot$ of $R$ (in particular, $L_R$) on
$\cz$.
\end{lemma}
\begin{proof} From the fibration $R/L_R\to L_P^2/L_R\to L_P^2/R$, we
get
\[H^0\Bigl( L_P^2/L_R,
\bc^d[\cz]^{\chi}\Bigr)\simeq H^0\Bigl( L_P^2/R, \bc[R/L_R]\otimes
(\bc^d[\cz]^{\chi})\Bigr).\] So, it suffices to define an $R$-module
isomorphism
\[\gamma: (\bc [S^*\backslash  S_1^* ] \otimes
\bc^d[\cz])^{\bc^*}\to  \bc[R/L_R]\otimes (\bc^d[\cz]^{\chi}).\] First,
define a morphism $\gamma_1: R/L_R\to S^*\backslash  S_1^*$ by
$(\gamma_1(rL_R))(X)=\chi (r)(r^{-1}X)_{1,1}$, for $r\in R$ and $X\in
S$. Then, $\gamma_1$ satisfies: \beqn \label{e6.-1}
\gamma_1(r'rL_R)=\chi(r')r'\cdot\gamma_1(rL_R),\,\,\text{for any}\,
r,r'\in R.\eeqn
 Now, define the morphism
\[\bar{\gamma}_1:R/L_R\times (\cz,\odot)\to ((S^*\backslash  S_1^*)\times
\cz)//\bc^*,\,\, (rL_R,x)\mapsto [\gamma_1(rL_R),x],\] where $(\cz,\odot)$
denotes the variety $\cz$ together with the action $\odot$ of $R$.
From ~\eqref{e6.-1}, it is easy to see that $\bar{\gamma}_1$ is an
$R$-equivariant morphism. Moreover, it is a biregular isomorphism.
(Observe that all the $\bc^*$-orbits in $(S^*\backslash S_1^*)\times
\cz$ are closed and hence $((S^*\backslash  S_1^*)\times \cz)//\bc^*$ is
the same as the orbit space $((S^*\backslash S_1^*)\times
\cz)/\bc^*$.) Now, $\gamma$ is nothing but the induced map from $\bar{\gamma}_1$.
\end{proof}

  Now, we determine $H^0 \bigl( L_P^2/L_R, \bc^d [\cz]^\chi\bigr)$.

  \begin{lemma} \label{6.4} For any $d\geq 0$,
    \beqn \label{e6.3}
H^0 \bigl( L_P^2/L_R, \bc^d [\cz]^\chi\bigr) \simeq \bigoplus_{\lam
=(\lam_1\geq \cdots \geq \lam_{n^2+1})\in D(L_P^2)} V_{L_P^2}(\lam)
\otimes\Hom_{L_R} (V_{L_P^2}(\lam), \bc^d[\cz]^\chi) .
    \eeqn
Thus, for any $\lam =(\lam_1\geq \lam_2\geq \cdots \geq
\lam_{n^2+1})\in D(L_P^2)$, $V_{L_P^2}(\lam)$ appears in $H^0 \bigl(
L_P^2/L_R, \bc^d [\cz]^\chi\bigr)$ if and only if  the following two
conditions are satisfied:
  \begin{enumerate}
\item $|\lam |=dm$, where $|\lam|:=\sum\lam_i,$ and
\item $\exists \,\mu =(\mu_1\geq \cdots \geq \mu_{n^2})$ such that $\mu$ interlaces $\lam$, i.e.,
\[ \lam_1 \geq \mu_1 \geq \lam_2 \geq \mu_2 \geq \cdots \geq \lam_{n^2}
\geq \mu_{n^2} \geq \lam_{n^2+1},
  \]
and the $GL(S_1 )$-irreducible module ${V}_{\GL(S_1)}(\mu)$ appears
in $\bc^d[\cz]$.
  \end{enumerate}
  \end{lemma}

  \begin{proof} The isomorphism \eqref{e6.3} of course follows from the Peter-Weyl theorem and the
  Tannaka-Krein duality.

   For $z\in\bc^*$, let
$\bar{z}$ be the diagonal matrix $[1, z, \dots, z] \in \Aut
S_1\subset \Aut S$ and $\hat{z}$ the diagonal matrix $[z, 1, \dots,
1]\in \Aut (\bc e_{1,1})\subset \Aut S$. Then, $\bar{z}\hat{z}$ acts
on $\cz$ via
  \beqn \label{e6.4}  (\bar{z}\hat{z} )\odot x = z^{n-m} (\bar{z}\cdot x)=z^{-m}x .
  \eeqn
By the branching law for the pair ($L_P^2=\GL(S), \GL(S_1)$) (cf.
[GW, Theorem 8.1.1]), we get, for any $\lam \in D(L_P^2)$, \beqn
\label{e6.5} V_{L_P^2}(\lam)\simeq \oplus_{\mu \in D(\GL(S_1)): \mu
\,\text{interlaces}\, \lam}\,\,V_{\GL(S_1)}(\mu),\,\,\,\text{as
$\GL(S_1)$-modules}. \eeqn Now, since $\GL(S_1)$ and
$\bar{z}\hat{z}$ generate the group $L_R$, combining the equations
~\eqref{e6.3}--\eqref{e6.5}, we get the second part of the lemma.
(Observe that the two actions $\cdot$ and $\odot$ of $\GL(S_1)$ on
$\cz$ coincide.)
  \end{proof}

  Combining Proposition ~\ref{6.2} with the Lemmas
  ~\ref{6.3}-- \ref{6.4} and the identities \eqref{e6.4}-- \eqref{e6.5},
  we get the following:

  \begin{theorem} \label{6.5} For any $d\geq 0$, decompose
  \[\bc^d[\cz]\simeq \oplus_{\mu\in D(\GL(S_1))}\,\, q_\mu(d)V_{\GL(S_1)}(\mu), \,\,\text{as $\GL(S_1)$-modules}.\]
  Then, as $L_P^2$-modules,
\begin{align}\label{e6.6} H^0\Bigl( & L_P^2/R, (\bc [S^*] \otimes
\bc^d[\cz])^{\bc^*}\Bigr) \simeq \notag\\ &\bigoplus_{\lam
=(\lam_1\geq \cdots\geq \lam_{n^2+1}\geq 0):|\lam|=dm}\,\,
\Bigl(\sum_{\mu=(\mu_1\geq \dots \geq \mu_{n^2}\geq 0): \mu
\,\text{interlaces}\, \lam} q_\mu(d)\Bigr) \, V_{L_P^2}(\lam).
  \end{align}
  In particular, $V_{L_P^2}(\lam)$ occurs in $H^0\Bigl( L_P^2/R, (\bc [S^*] \otimes
\bc^d[\cz])^{\bc^*}\Bigr)$ if and only if the following two conditions
are satisfied:
  \begin{enumerate}
  \item  $\lam =
(\lam_1\geq \cdots \geq \lam_{n^2+1}\geq 0)$ and  $|\lam |=dm$, and
  \item there exists a $ \mu = (\mu_1\geq \cdots \geq \mu_{n^2}\geq 0)$ which interlaces $\lam$
and such that  the irreducible $GL(S_1)$-module $V_{\GL(S_1)}(\mu)$
occurs in $\bc^d[\cz].$
\end{enumerate}

(Observe that if $V_{\GL(S_1)}(\mu)$ occurs in $\bc^d[\cz],$ then
automatically $|\mu |=dn$ and $\mu_{n^2}\geq 0$, since $\bc^d[\cz]$
is a $\GL(S_1)$-module quotient of $S^d(S^n(E))$.)
\end{theorem}
\begin{remark} {\em Since
\[(\bc[S^*]\otimes \bc^d[\cz])^{\bc^*}\simeq S^{(m-n)d}(S)\otimes
\bc^d[\cz],\]
and $S$ is a $L_P^2$-module, we also get (using [Ku1, Lemma 8])
\begin{align*}H^0\Bigl( L_P^2/R, (\bc [S^*] \otimes
\bc^d[\cz])^{\bc^*}\Bigr)&\simeq S^{(m-n)d}(S)\otimes
H^0(L_P^2/R, \bc^d[\cz])\\
&\simeq \oplus_{\mu =(\mu_1\geq \dots \geq \mu_{n^2}): \mu_{n^2}\geq
0} \,\,q_\mu(d) S^{(m-n)d}(S)\otimes V_{L_P^2}(\hat{\mu}),
\end{align*}
where $\hat{\mu}:= (\mu_1\geq \dots \geq \mu_{n^2}\geq 0)\in
D(L_P^2).$}
\end{remark}

\section{nonnormality of the orbit closures of $\fp$}

It is easy to see that the morphism $\al$ of Section 5 induces an
injective map (for any $d \geq 0$)
\[\al^*:\bc^d[\overline{R\cdot \fp}] \hookrightarrow
(\bc [S^*] \otimes \bc^d[\cz])^{\bc^*}=S^{d(m-n)} (S) \otimes
\bc^d[\cz].\]

 \begin{proposition} \label{7.1} For any $m\geq 2n$, the inclusion
    \[
H^0\bigl( L_P^2/R, \bc^d [\overline{R\cdot \fp}]\bigr)
\hookrightarrow H^0\bigl( L_P^2/R, (\bc [S^*] \otimes
\bc^d[\cz])^{\bc^*}
\bigr),
  \]
induced from the inclusion $\al^*$, is not an isomorphism for $d=1$.
  \end{proposition}
\begin{proof}
Of course,
$\bc^1 [\overline{R\cdot \fp}]$  is a $R$-module quotient of
$S^m(E)$; in fact, it is a $R$-module quotient of
$S^m(S).$
 Let $K$ be the kernel
  \beqn \label{e7.1} 0\to K\to S^m(S) \to \bc^1[\overline{R\cdot \fp}]
  \to 0.
  \eeqn
We first determine the linear span $\ip<\overline{R\cdot \fp}>$ of
the image of $\overline{R\cdot \fp}$ inside $S^m(S^*)$.  For $u\in
U_R, z\in \bc^*$ and $g\in\GL (S_1)$ (where $\tau_z\in \Aut (\bc e_{1,1})$
is defined by $\tau_z(e_{1,1})=ze_{1,1}$),
  \begin{align*}
\bigl( (gu\tau_z)^{-1}\cdot \fp\bigr) &(x_{1,1}e_{1,1} + \sum_{m-n+1\leq i,j\leq m}\, x_{i,j}
e_{i,j})\\
 &= \fp\bigl( (z x_{1,1} +\sum x_{i,j}a_{i,j}) e_{1,1} + g\sum x_{i,j}e_{i,j}
 \bigr) \\
  &\qquad\quad (\text{where }ue_{i,j} = e_{i,j}+a_{i,j}e_{1,1}) \\
 &= (zx_{1,1} + \sum x_{i,j}a_{i,j})^{m-n} (g^{-1}\cdot\perm )
 (\sum x_{i,j}e_{i,j}) .
   \end{align*}
For any vector space $V$, the span of $\{v^{m-n}, v\in V\}$ inside
$S^{m-n}(V)$ equals $S^{m-n}(V)$. Moreover, since $S^n(S_1^*)$ is an
irreducible $\GL(S_1)$-module, the span of $\{g^{-1}\cdot\perm
\}_{g\in \GL(S_1)}$ is equal to $S^n(S_1^*).$
 Here we
have identified $S^n(S_1^*) \hookrightarrow S^n(S^*)$ via the
projection
  $S \to S_1,\,\,
  e_{1,1} \mapsto 0.$

  Thus,
  \begin{align*}
\ip<\overline{R\cdot \fp}> &= S^n(S_1^*)\cdot S^{m-n}(S^*)\\
&= \lam_o^{m-n} S^n(S_1^*)\oplus \lam_o^{m-n-1} S^{n+1}(S_1^*)\\
&\qquad \oplus\; \cdots \, \oplus\; \lam_o^0 \; S^m(S_1^*),
  \end{align*}
  where $\lam_o\in S^*$ is defined in the proof of Proposition ~\ref{5.3}.
Thus,
  \[  K \simeq e^{m-n+1}_{1,1} S^{n-1}(S_1)\oplus\cdots\oplus e_{1,1}^m
  S^0(S_1).
   \]
None of the weights of $K$ are $L_P^2$-antidominant with respect to
the basis $\{e_{1,1}, e_{i,j}\}_{m-n+1\leq i,j\leq m}$ of $S$ if
  \[  m-n+1 > n-1,\quad\, \text{i.e., if}\, m > 2n-2.  \]
Hence,
  \beqn \label{e7.2} H^0( L_P^2/R, K) = 0, \quad \text{if}\, m> 2n-2.
  \eeqn
Also,
  \beqn \label{e7.3}
H^1 ( L_P^2/R, K) = 0, \quad\text{ if } m>2n-1.
  \eeqn
To prove this, it suffices to show that, for any weight $\mu$ of $K$
and any simple reflection $s_i$ for $L_P^2$, $s_i(-\mu +\rho )-\rho$ is not
dominant, i.e.,  $s_i\mu +\al_i$ is not antidominant.
Writing $\mu = (\mu_1, \dots, \mu_{n^2+1})$, we have
\[
\mu_1 > \mu_j+1,\quad\; \forall j\geq 2 \,\,\text{(since $m> 2n-1$)} .
  \]
Thus, if $i>1$,
  \[
(s_i\mu +\al_i)_1 = \mu_1 > (s_i\mu +\al_i)_2.
  \]
Hence, $s_i\mu +\al_i$ is not antidominant for $i>1$.  For $i=1$, we get
\[
(s_1\mu +\al_1)_2=\mu_1-1 > (s_1\mu +\al_1)_3=\mu_3.\]
Combining \eqref{e7.2}--\eqref{e7.3}, we get
  \beqn \label{e7.4}
H^0( L_P^2/R, K) = H^1 ( L_P^2/R, K) = 0\,\,\text{for all}\, m\geq 2n.
   \eeqn
   Considering the long exact cohomology sequence, corresponding to the
   coefficient sequence  ~\eqref{e7.1}, we get for all $m\geq 2n$ (by using
   ~\eqref{e7.4}),
  \beqn \label{e7.4'}
H^0( L_P^2/R, \bc^1[\overline{R\cdot \fp}])\simeq H^0(L_P^2/R, S^m(S))
 =S^m(S).
  \eeqn
  In particular, $H^0( L_P^2/R, \bc^1[\overline{R\cdot \fp}])$ is an irreducible
  $L_P^2$-module.

  Next, we determine $M=H^0(L_P^2/R, (\bc[S^*]\otimes \bc^1[\cz])^{\bc^*}).$
By Theorem ~\ref{6.5}, the irreducible $L_P^2$-module
$V_{L_P^2}(\lam)$ appears in $M$ if and only if the following three conditions
are satisfied:

1) $\lam_{n^2+1}\geq 0$, \, $|\lam |=m$,

2) $\exists \,\mu =(\mu_1\geq \cdots \geq \mu_{n^2}\geq 0)$
which interlaces $\lam$, and

3) the irreducible $\GL(S_1)$-module $V_{\GL(S_1)}(\mu)$ occurs in $\bc^1
[\cz]$.

But, $\bc^1[\cz]$ is the irreducible $\GL(S_1)$-module $S^n(S_1)$,
since $\cz$ is a closed $\GL(S_1)$-subvariety of $S^n(S_1^*)$. Thus,
$\mu =(n\geq 0\geq 0\geq \cdots \geq 0)$.  Hence, $V_{L_P^2}(\lam)$
occurs in $M$ if and only if
  \[
\lam=(\lam_1\geq\lam_2 \geq 0 \cdots \geq 0)\,\, \text{with} \,
\lam_1\geq n\geq \lam_2 \,\, \text{and} \, \lam_1 +\lam_2 = m.
  \]
In particular, $M$ is not irrreducible. This proves the proposition.
    \end{proof}

  \begin{corollary} \label{7.2} Let $m\geq 2n$. Then,  $\overline{R\cdot \fp}$ is
  {\em not} normal.
  \end{corollary}
\begin{proof} If $\overline{R\cdot \fp}$ were normal, by
  the original form of the Zariski's main theorem
  (cf. [M, Chap. III, $\S$9]) and Proposition ~\ref{5.3}
  (following its notation),
    \[
\al^*:\bc [\overline{R\cdot \fp}] \to
\bc [(S^*\times \cz)//\bc^*]
  \]
would be an isomorphism.  In particular, we would get the $R$-module
ismorphism
  \[
\al^*:\bc^1[\overline{R\cdot \fp}] \simto (\bc[ S^*]\otimes
\bc^1[\cz])^{\bc^*} .
  \]
But this contradicts Proposition ~\ref{7.1}.
  \end{proof}
The following corollary follows similarly.
  \begin{corollary} \label{7.5} Let $m\geq 2n$. Then,  $\overline{L_P^2\cdot \fp}$ is
  {\em not} normal.
  \end{corollary}

  \begin{proof}  By Definition ~\ref{5.1} and Lemma ~\ref{5.2}, we have
  the proper, surjective, birational morphism
  \[
\beta: L_P^2 \times_R (\overline{R\cdot \fp}) \to
\overline{L_P^2\cdot \fp}.
  \]
If $\overline{L_P^2\cdot \fp}$ were normal, both the maps $\beta$
and the
  composite map $\beta\circ (\text{Id}\times \al)$ (which are both proper and
  birational morphisms)
  \[
L_P^2 \times_R \bigl( (S^*\times \cz)//\bc^*\bigr)
\overset{\text{Id}\times \al}{\longrightarrow} L_P^2
\times_R (\overline{R\cdot \fp})\overset{\beta}{\twoheadrightarrow}
 \overline{L_P^2 \cdot \fp}
  \]
  would induce isomorphisms (via the Zariski's main theorem
  [H, Chap. III, Corollary 11.4 and its proof])
\[
\beta^*:\bc \bigl[ \overline{L_P^2 \cdot \fp}\bigr] \simto
H^0\bigl( L_P^2/R, \bc [\overline{R\cdot \fp}]\bigr)\]
and
\[\bigl(\beta\circ (\text{Id}\times \al)\bigr)^*:\bc \bigl[ \overline{L_P^2 \cdot \fp}\bigr]
\simto H^0\bigl( L_P^2/R, \bc\bigl[S^*\times \cz\bigr]^{\bc^*}
\bigr).
  \]
In particular, the canonical map
\[(\text{Id}\times \al)^*:H^0(L_P^2/R, \bc
\bigl[ \overline{R\cdot \fp}\bigr]) \simto H^0\bigl( L_P^2/R,
\bc\bigl[S^*\times \cz\bigr]^{\bc^*} \bigr)
  \]
  would be an isomorphism. This contradicts Proposition ~\ref{7.1}.
   Hence $\overline{L_P^2\cdot
\fp}$ is not normal.
  \end{proof}

  \begin{theorem}  \label{7.4} Let $m\geq 2n$. Then, $\overline{{G}\cdot \fp}$ is {\em not} normal.
    \end{theorem}

  \begin{proof}  Recall from Section 4 the proper and surjective morphism
  $\phi : {G} \times_{{P}} (\overline{{P}\cdot \fp}) \twoheadrightarrow \overline{{G}\cdot
  \fp}.$ It is birational by Corollary ~\ref{4.4}.
Consider the projection $\pi:P \to L_P^2,$ obtained by identifying
$L_P^2\simeq P/(U_P\cdot L_P^1)$ and let $P_R$ be the parabolic
subgroup of $P$ defined as $\pi^{-1}(R).$ Now, define the variety
\[Y=P\times_{P_R}\bigl((S^*\times \cz)//{ \bc^*}\bigr),\]
where $P_R$ acts on $(S^*\times \cz)//{\bc^*}$ via its projection
onto $R$. Consider the morphism
\[\al_P:Y\to \overline{P\cdot \fp}= \overline{L_P^2\cdot
\fp},\,\,[p,x]\mapsto p\cdot\al(x),\] for $p\in P$ and $x\in
(S^*\times \cz)//{\bc^*}.$ Observe that, under the canonical
identification (induced from the map $\pi$)
$L_P^2\times_{R}\bigl((S^*\times \cz)//{\bc^*}\bigr)\simeq Y$, the
map $\al_P$ is nothing but the composite map $\beta\circ
(\text{Id}\times \al)$ (cf., the proof of Corollary ~\ref{7.5}).
Hence, $\al_P$ is a proper, birational morphism. The $P$-morphism
$\al_P$ of course gives rise to a proper, birational $G$-morphism
\[\bar{\al}_P:G\times_P Y\to G\times_P (\overline{P\cdot \fp}).\]
Finally, define the proper, birational, surjective $G$-morphism as
the composite
\[\hat{\al}_P:=\phi\circ \bar{\al}_P:G\times_P Y\to  \overline{G\cdot \fp}.\]
If $ \overline{G\cdot \fp}$ were normal, we would get an isomorphism
\[{\hat{\al}_P}^*:\bc[\overline{G\cdot \fp}]\to \bc[G\times_P Y]\simeq H^0(G/P,
H^0(L_P^2/R, \bc[S^*\times \cz]^{\bc^*})),  \] where $P$ acts on
$H^0(L_P^2/R, \bc[S^*\times \cz]^{\bc^*})$ via its projection $\pi$.
It is easy to see that this, in particular, would induce an
isomorphism
\beqn \label{e7.6} \bc^1[\overline{G\cdot \fp}]\simeq
H^0(G/P, H^0(L_P^2/R, \bigl(\bc[S^*]\otimes
\bc^1[\cz]\bigr)^{\bc^*})). \eeqn
Now, by the proof of Proposition
~\ref{7.1}, there exists $k_\lam>0$ such that
\begin{align*}
H^0(& G/P, H^0(L_P^2/R, \bigl(\bc[S^*]\otimes
\bc^1[\cz]\bigr)^{\bc^*}))\\
&\simeq \oplus_{\lam=(\lam_1\geq \lam_2\geq 0\geq\cdots \geq 0)\in
D(L_P^2):\lam_1\geq n\geq \lam_2, \lam_1+\lam_2=m}\,\,
k_\lam H^0(G/P, V_{L_P^2}(\lam))\\
&\simeq \oplus_{\hat{\lam}=(\lam_1\geq \lam_2\geq 0\geq\cdots \geq
0)\in D(G):\lam_1\geq n\geq \lam_2,\lam_1+\lam_2=m}\,\, k_\lam
V_{G}(\hat{\lam}),\,\,\text{by [Ku1, Lemma 8]},
\end{align*}
 where $\hat{\lambda}$ is obtained from $\lambda$ by adding
 $m^2-n^2-1$ zeroes in the end to $\lambda$. In particular, $H^0(G/P, H^0(L_P^2/R, \bigl(\bc[S^*]\otimes
\bc^1[\cz]\bigr)^{\bc^*}))$ is not an irreducible $G$-module.

Finally, $\bc^1[\overline{G\cdot\fp}]$ is, by definition, a
$G$-module quotient of the irreducible $G$-module $Q^*\simeq S^m(E)$.
Clearly,  $\bc^1[\overline{G\cdot\fp}]$ is nonzero and hence
\[\bc^1[\overline{G\cdot\fp}]\simeq S^m(E).\]
This contradicts \eqref{e7.6} and hence the theorem is proved.
  \end{proof}

\bibliographystyle{plain}
\def\noopsort#1{}

\vskip6ex

\noindent Address:
 Department of Mathematics, University of North Carolina, Chapel Hill, NC 27599-3250, USA\\
\noindent(email: shrawan@email.unc.edu)

  \end{document}